\documentclass[10pt]{amsart}
\usepackage{custom_anna}
\hypersetup{
 pdfauthor = {Anna Fokma},
 pdftitle= {Thurston's jiggling},
 pdfsubject = {...}}

\newcommand{\aspan}{\operatorname{ASpan}}
\newcommand{\lspan}{\operatorname{LSpan}}
\newcommand{\bary}{\mathrm{bar}}
\newcommand{\LinMaps}{\operatorname{Lin}}
\newcommand{\Bop}[1]{B_{\mathrm{op}}^{#1}}
\newcommand{\dop}[1]{d_{\mathrm{op}}^{#1}}
\newcommand{\dproj}{d_\mathrm{proj}}
\newcommand{\Bproj}{B_\mathrm{proj}}
\newcommand{\maxcoeff}{{\Lambda}}

\begin{document}

\title{Thurston's jiggling}

\date{\today}


\author{Anna Fokma}
\address{Utrecht University, Department of Mathematics, Budapestlaan 6, 3584~CD Utrecht, The Netherlands}
\email{annaf.math@outlook.com}

\begin{abstract}
In the 1970s Thurston introduced a technique known as ``jiggling'' which brings any triangulation into general position (a stronger version of transversality) by subdividing and perturbing. This result is now known as Thurston's jiggling lemma. In this paper we provide an alternative, more conceptual proof of the lemma. In particular we also prove the generalization to manifolds, whose proof had previously only been sketched.
\end{abstract}
\maketitle

\tableofcontents

\section{Introduction}\label{chap:JigglingIntro}

A crucial step in Thurston's work~\cite{Th1,Th2} on determining the homotopy type of the space of foliations, is to bring triangulations in general position with respect to a distribution. For readers unfamiliar with this notion, it can be thought of as requiring not only that all top-dimensional simplices of the triangulation are transverse to the distribution, but all their lower-dimensional faces as well. This implies in particular that the distribution varies only little over each simplex. To achieve this, Thurston established what is now known as the ``jiggling lemma'':

\begin{lemma*}[Thurston's jiggling lemma \cite{Th2}]
	Consider a manifold $M$ endowed with a distribution $\xi$. Any smooth triangulation of $M$ can be subdivided and subsequently perturbed to be in general position with respect to $\xi$, over any given compact subset. Moreover, the perturbation can be assumed to be $C^1$-small. 
\end{lemma*}

Thurston considered the jiggling lemma to be intuitive enough to be accepted without a formal proof. However, he did provide one for those interested, and upon inspection, it becomes clear that the proof is quite subtle. Recently a variation of Thurston's jiggling lemma has been established in \cite{BD24}, where the triangulation is brought into general position with respect to a symplectic form. More generally, the lemma has had numerous applications in the study of geometric structures~\cite{LM,colin1999stabilite,Vo16,CPPP,PV}.

In the present paper we revisit Thurston's jiggling lemma and try to develop a more conceptual alternative to Thurston's proof, which is at times quite technical. Moreover, Thurston proves his jiggling lemma in the case where $M$ is Euclidean in detail, but only sketches the argument for general manifolds $M$. Here we also provide a full proof for general $M$, for which we develop the notion of relative jiggling.

\subsection{Thurston's jiggling} \label{sec:IntroThurston}

Before we start with our approach to jiggling, we first briefly recall Thurston's approach. A detailed account of his approach is given in \cite{ben97}. Thurston's argument for his jiggling lemma depends on two main ingredients: subdividing and perturbing. We have illustrated the proof in \cref{fig:introJiggling}.

\begin{figure}[h]
	\centering
	\begin{subfigure}[b]{0.3\textwidth}
		\centering
		\includegraphics[width=\textwidth,page=1,trim =0 3cm 0 3cm,clip]{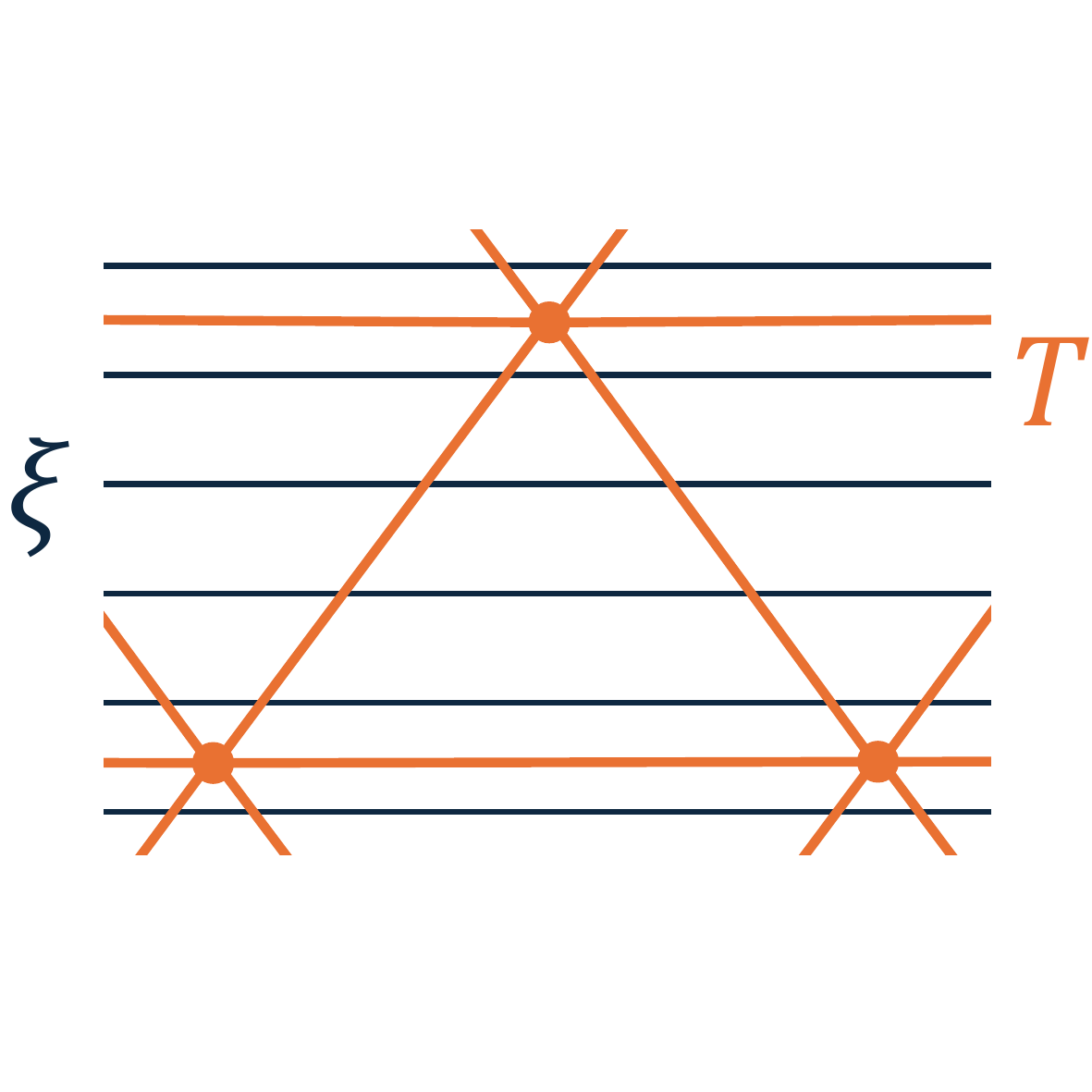}
	\end{subfigure}
	\raisebox{1cm}{\begin{tikzpicture}
			\draw[thick,->] (0,0) -- ++ (0.03\textwidth,0);	
	\end{tikzpicture}}
	\begin{subfigure}[b]{0.3\textwidth}
		\centering
		\includegraphics[width=\textwidth,page=2,trim =0 3cm 0 3cm,clip]{Fig_introjiggling}
	\end{subfigure}
	\raisebox{1cm}{\begin{tikzpicture}
			\draw[thick,->] (0,0) -- ++ (0.03\textwidth,0);
	\end{tikzpicture}}
	\begin{subfigure}[b]{0.3\textwidth}
		\centering
		\includegraphics[width=\textwidth,page=3,trim =0 3cm 0 3cm,clip]{Fig_introjiggling}
	\end{subfigure}
	\centering
	\caption{On the left we are given the horizontal distribution $\xi$ on $\R^2$ and a triangulation $T$ of $\R^2$. The triangulation $T$ is not transverse to the distribution, and hence we jiggle $T$. We do so in two steps: we first subdivide the triangulation (middle figure), and then we perturb the image of $T$ on the vertices (right figure). The result is a transverse triangulation. } \label{fig:introJiggling}
\end{figure}

The proof starts by subdividing the triangulation in a crystalline manner. This ensures in particular that the size of the simplices decreases uniformly. 
Hence by subdividing, we can ensure that the distribution becomes almost constant over each simplex.

The next step is to inductively perturb the vertices of the triangulation such that the simplices they span are transverse. Here, transverse is meant for each simplex with respect to a constant foliation approximating the distribution over that simplex. This is the main technical part of the proof, where Thurston shows that some uniform amount of transversality to the constant foliation can be achieved, independent of the number of subdivisions, using various bounds on the size of the simplices and distances in the Grassmannian. 

To conclude, Thurston needs to deduce transversality to the distribution itself. For this he uses that the obtained transversality is uniform and that hence, by taking the number of subdivisions large enough, the transversality of the simplices to the constant foliation overcomes the variation of the distribution within the simplex. It follows that the triangulation is transverse (and moreover, in general position).

Thurston first deals with all of the above in the case where the triangulated manifold $M$ is Euclidean. Then, he sketches the proof for a general manifold $M$ in an extrinsic manner by embedding the manifold $M$ into Euclidean space, jiggling the triangulation and projecting it back to $M$. 

\subsection{A variation of Thurston's proof} \label{sec:IntroThJiggling}
We now discuss our point of view on Thurston's jiggling and how it differs from Thurston's original proof. 

Throughout Thurston's proof, the notion of jiggling remains rather implicit. To make it explicit, we first recall that a triangulation is a piecewise smooth map defined on a linear polyhedron such that it is both a piecewise smooth embedding and a homeomorphism. A polyhedron $|K|$ is the union of the linear simplices in a simplicial complex $K$ in Euclidean space. We observe that for jiggling it is useful to separate the map from the simplicial complex. This brings us to the following. 
\begin{definition}\label{def:ThJiggling}
	Let $N$ be a smooth manifold and let $f:|K| \to N$ and $f': |K'| \to N$ be piecewise smooth maps with respect to finite simplicial complexes $K$ and $K'$ respectively. We say that $(f',K')$ is an \textbf{$\epsilon$-jiggling} of $(f,K)$ if $K'$ is a subdivision of $K$ and $d_{C^1}(f,f') < \varepsilon$, where $\varepsilon \in \R_{>0}$.
\end{definition}
With this separation of map and simplicial complex, we can jiggle a triangulation $T: |K| \to N$, by subdividing $K$ and perturbing the map $T$.

As does Thurston, we subdivide our simplicial complex in a crystalline manner such that the distribution becomes almost constant. We also perturb the vertices inductively such that the simplices they span are transverse. However, our approach to choosing the perturbation diverges from Thurston's. Fundamentally, our arguments reduce to the following elementary observation: a $1$-simplex $\langle p,q \rangle$ in $\R^n$ is transverse to a constant foliation $\fol$ of rank $k<n$ if and only if $p$ and $q$ do not agree under the projection $\R^n \to \R^n / \fol \simeq \R^{n-k}$. The transversality of a $d$-simplex $\join{p}{\Delta}$ similarly follows if the $(d-1)$-simplex $\Delta$ is transverse and if the projection of $p$ and the affine span of $\Delta$ do not intersect in $\R^{n-k}$. Then by just using the compactness of the Grassmannian and the size of the simplices, we obtain a uniform bound on the transversality we achieve with respect to a constant foliation approximation the distribution. As in Thurston's proof this then implies transversality to the distribution itself.

Similar to Thurston, we first deal with the case where the codomain of the map $f$ is Euclidean. We however prefer an intrinsic proof in the case where the codomain is a manifold. Hence we develop a method of jiggling a function $f: |K| \to N$ relative to a region where it is already in general position, which allows us to deduce the manifold case by jiggling chart by chart relative to the previous charts. This needs to be done carefully, as general position is generally not preserved under subdivisions. Hence here we first jiggle the identity map $|K| \to |K|$ to a map $\iota':|K|\to|K|$ while making sure the image of $\iota'$ is also $|K|$. Moreover, we achieve that $\iota'$ is such that we are able to subdivide $K$ in a crystalline manner while preserving the general position of the composition $|K|\stackrel{\iota'}{\to}|K|\stackrel{f}{\to} N$. This then allows us to jiggle the map $|K|\to N$ as in the non-relative case.

The main theorem is hence the following, which implies Thurston's jiggling lemma. 

\begin{restatable*}{thm}{TjigglingMfd}
	\label{th:TjigglingMfd}
	Consider a finite simplicial complex $K$ and a manifold $N$ endowed with a distribution $\xi$. Then, given
	\begin{itemize}
		\item $\gamma>0$,
		\item a piecewise embedding $f:|K| \to N$, 
		\item a subcomplex $A$ of $K$ such that $(f|_{|A|},A)$ is in general position with respect to $\xi$,
	\end{itemize}
	there exists a $\gamma$-jiggling $(g,K')$ of $(f,K)$ such that
	\begin{itemize}
		\item $(g,K')$ is in general position with respect to $\xi$, and 
		\item $g|_{|A|} = f|_{|A|}$.
	\end{itemize}
\end{restatable*}

%%%%%%%%%%%%%%%%%%%%%%%%%%%%%%%%%%%%%%%%%%%%
\subsection{Outline}\label{sec:Outline}
In \cref{sec:prelim} we discuss the preliminaries, including simplicial complexes, piecewise maps and triangulations. In \cref{sec:crysSubdiv} we discuss crystalline subdivisions of triangulations. In \cref{sec:linearizing} we show how we can (locally) approximate a map or section by its linearization, reducing the jiggling of general maps and sections to the jiggling of piecewise linear maps and sections. Most of \cref{sec:prelim,sec:crysSubdiv,sec:linearizing} appears already in \cite{FPTjiggling}, which is joint work of the author with \'Alvaro del Pino and Lauran Toussaint. Thus, we skip most of the proofs here, but they can all be found in \cite{FPTjiggling}.

In \cref{sec:prelimTransv} we discuss various notions of transversality, including general position and semitransversality. We relate some of these notions in \cref{sec:transEstimates}, where we also establish alternative criteria to check transversality. Additionally, we discuss how to reduce the study of transversality with respect to (not necessarily constant) distributions to the case of constant foliations. In \cref{sec:PerturbVertex} we consider the situation where we are given a point $p$ and a set $\SD$ of linear simplices that are transverse to a constant foliation $\fol$. Then, we discuss how to perturb $p$ to $p'$ such that the simplices spanned by $p'$ and $\SD$ are also transverse to $\fol$. For this, the quantitative notion of transversality we introduce in \cref{sec:prelimTransv}, semitransversality, is particularly useful. Then in \cref{sec:Tjiggling} we jiggle using the methods developed in \cref{sec:PerturbVertex}, recovering Thurston's jiggling lemma.

\subsection{Acknowledgements}
The author wants to thank \'Alvaro del Pino and Lauran Toussaint for numerous discussions on jiggling and for reading multiple drafts of this paper. This paper forms a part of the author's PhD thesis. Hence, the author would like to thank the reading committee, consisting of M\'{e}lanie Bertelson, Kai Cieliebak, Ga\"{e}l Meigniez, Ieke Moerdijk and Thomas Rot, for reviewing the thesis and offering their feedback.

%%%%%%%%%%%%%%%%%%%%%%%%%%%%%%%%%%%%%%%%%%%%%%%%%%%%

\section{Preliminaries} \label{sec:prelim}

In this section we cover some basics and fix notation. We discuss simplicial complexes in \cref{ssec:simplicialComplexes}, piecewise linear and smooth maps in \cref{sec:piecewiseMaps} and we end with triangulations in \cref{sec:triangul}.

%%%%%%%%%%%%%%%%%%%%%%%%%%%%%%%%%%%%%%%%%%%%%%%%%%%%%%
\subsection{Simplicial complexes} \label{ssec:simplicialComplexes}

We define the standard $m$-dimensional simplex $\Delta^m \subset \R^m$ as 
\begin{equation*}
	\Delta^m = \{(t_1,\dots,t_m) \in \R^m \mid \sum_{i=1}^m t_i \leq 1 \text{ and } t_i \geq 0 \text{ for all } i \}.
\end{equation*}
A \textbf{linear simplex} is then any subset of Euclidean space that is affinely isomorphic to a standard simplex. Given a set $\{p_0,\dots,p_{m}\}$ of points in $\R^N$, we denote by $\langle p_0,\dots,p_m \rangle$ the linear simplex they span in $\R^N$. By a \textbf{face} of a linear simplex we refer to a subsimplex of any dimension.

Simplices can be glued along their codimension-1 faces if they form a so-called simplicial complex:
\begin{definition}
A \textbf{simplicial complex} is a locally finite set $K$ of linear simplices in an Euclidean space such that
	\begin{itemize}
		\item if $\sigma \in K$ then its faces are also in $K$, and
		\item if $\sigma,\sigma'\in K$ then $\sigma \cap \sigma'$ is either empty or a face of both $\sigma$ and $\sigma'$.
	\end{itemize}
\end{definition}
We denote by $K^{(\topd)}$ the set of top-dimensional simplices of a simplicial complex $K$.

A linear polyhedron is the topological space that is spanned by a simplicial complex. We introduce the word ``linear'' to make a distinction with polyhedra living in arbitrary manifolds, although we will not use these explicitly.
\begin{definition}
	A \textbf{linear polyhedron} $P$ is the union of the linear simplices in a simplicial complex $K$. That is, $P = \cup_{\Delta \in K} \Delta$. In this case, we call $K$ a \textbf{triangulation} of $P$ and write $P = |K|$.
\end{definition}
We will often work with simplicial complexes of \textbf{pure} dimension. That is, every simplex should be contained in a simplex of top dimension.

We can subdivide any simplicial complex to obtain a new simplicial complex triangulating the same polyhedron. In particular, we remark that any two triangulations of a given polyhedron have a common subdivision.
\begin{definition}
	A \textbf{subdivision of a simplicial complex} $K$ is a simplicial complex $K'$ such that each $\Delta \in K$ satisfies $\Delta= \cup_{i \in I} \Delta'_i$ for a finite collection of $\Delta'_i \in K'$.
\end{definition}

\subsubsection{Adjacency} \label{sec:prelimAdjacency}

The appropriate notion of a neighborhood of a linear simplex in a simplicial complex is that of a star, which is defined using the concept of adjacency. Related is the notion of a ring.
\begin{definition} \label{def:starRing}
Consider a simplicial complex $K$ including a linear simplex (or more generally, a subcomplex) $Q$.
\begin{itemize}
\item Two linear simplices are \textbf{adjacent} if they share a face.
\item The \textbf{star} $\str{Q}$ of $Q$ is the set of all its adjacent simplices and their faces.
\item The \textbf{closure} $\cl(A)$ of a subset $A$ of a simplicial complex is the smallest subcomplex containing the subset $A$.
\item We define the \textbf{ring} around $Q$ as $\ring{Q} = \cl(\str (Q) \setminus Q)$.
\end{itemize}  
\end{definition}
We write $\strT[K]{Q}$ and $\ringT[K]{Q}$ if we want to emphasize the simplicial complex we are working with. The $n$-fold iteration of $\str$ is denoted by $\strN{n}$.

We also introduce the set of vertices $w$ that neighbor a given vertex $v$, in the sense that both $v$ and $w$ are adjacent to the same edge, using the notion of a link. 
\begin{definition}
	Consider a simplicial complex $K$ with a vertex $v\in K$.
	\begin{itemize}
		\item The \textbf{link} of $v$, denoted $\link(v)$, is the set of all simplices $\sigma \in K$ such that $v\notin \sigma$ and such that the simplex spanned by $v$ and $\sigma$ is an element of $K$.
		\item The \textbf{vertex-link} of $v$, denoted $\vlink(v)$, is the set of vertices in $\link(p)$.
	\end{itemize}
\end{definition}

\subsection{Piecewise maps} \label{sec:piecewiseMaps}

When working with maps defined on polyhedra, it is natural to consider either piecewise linear or piecewise smooth maps. We discuss both, although we first recall the notion of the first jet bundle to be able to endow the space of such maps with a $C^1$-metric. 

%%%%%%%%%%%%%%%%%%%%%%%%%%%%%%%%%%%%%%%%%%%%%%%%%%%%%%
\subsubsection{Jet bundles} \label{sec:JetHPrinc}

In this section we briefly recall the standard terminology on (first order) jet bundles. For a more elaborate overview we refer to \cite{CiElMi,Gr86}.

Given two manifold $M$ and $N$, the \textbf{first jet bundle} of maps $M \to N$ is the bundle $J^1(M,N) \to M$ of first order Taylor polynomials of maps $M \to N$. The projection is defined by sending a Taylor polynomial at $x\in M$ to the point $x$. We observe that such a Taylor polynomial at $x$ can be represented by a map locally defined around $x$. Vice versa, every map $f : M \to N$ induces a section $j^1 f : M \to J^1(M,N)$.

\subsubsection{Piecewise linear/smooth maps} \label{sec:PLSmaps}

Between linear polyhedra, the natural classes of maps to consider are:
\begin{definition}
Let $P$ be a linear polyhedron. A map $f: P \rightarrow \R^n$ is \textbf{piecewise linear/smooth} if:
\begin{itemize}
    \item the map $f$ is continuous, and
    \item for some triangulation $K$ of $P$, the maps $f|_\Delta$ are smooth for all $\Delta \in K$.
\end{itemize}
\end{definition}
If we want to emphasize the role of $K$, we will write that $(f,K)$ is a piecewise linear/smooth map or that $f$ is piecewise linear/smooth with respect to $K$. We denote the set of piecewise linear maps by $\MapsPL(P,\R^n)$ and the set of piecewise smooth maps by $\MapsPS(P,N)$. 
We note that $\MapsPS(P,\R^n)$ contains $\MapsPL(P,\R^n)$.

The set $\MapsPS(P,N)$ can readily be endowed with the $C^0$-topology (either the weak or the strong version), by interpreting it as a subset of $C^0(P,N)$. 
To define the $C^1$-topology in $\MapsPS(P,N)$, it is convenient to assume that $P$ is of pure dimension. We fix an auxiliary triangulation $K$ of $P$ to make sense of $J^1(\Delta,N)$ for each top simplex $\Delta \in P$. On each $J^1(\Delta,N)$ we fix a metric, so that we can make sense of the $C^1$-metric on $\MapsPS(\Delta,N)$. Then we define:
\begin{definition} \label{def:CrTopPolyhedron}
Consider a pair of maps $f_1, f_2 \in \MapsPS(P,N)$, each piecewise smooth with respect to a triangulation $K_i$ of a compact polyhedron $P$, for $i=1,2$. Consider a triangulation $K'$ subdividing $K_1$, $K_2$ and $K$. We write $\dist{1}(f_1,f_2) < \epsilon$ if for every top-dimensional simplex $\Delta \in K'$ we have $\dist{1}(f_1|_\Delta,f_2|_\Delta) < \epsilon$, where $\epsilon \in \R_{>0}$. 
\end{definition}
We observe that the above distance on $\MapsPS(P,N)$ depends on our choice of metrics on each $J^1(\Delta,N)$ and therefore it also depends on $K$. It does however not depend on the choice of $K'$. The underlying topologies (known as respectively the weak and strong) do not depend on any of these choices.

%%%%%%%%%%%%%%%%%%%%%%%%%%%%%%%%%%%%%%%%%%%%%%%%%%%%%%%%%%%%%%%%%%%%%%%%%%
\subsection{Triangulations} \label{sec:triangul}

We now recall the notion of a triangulation of a manifold, which we think of as a decomposition of the manifold into smooth simplices. This combinatorial description of the manifold is particularly useful for local arguments and, as we shall see, will turn our arguments into arguments about linear polyhedra. 
For a more detailed account of triangulations we refer to \cite[Ch. IV.B]{whitney2012geometric} and \cite{lurie2009topics}.

\begin{definition}
	A \textbf{triangulation}  $T: |K| \rightarrow M$ of a manifold $M$ consists of a simplicial complex $K$ and a family of smooth embeddings $(T_\Delta : \Delta \to M)_{\Delta \in K}$ that glue to a homeomorphism $T: |K| \to T(|K|)$ satisfying $T(|K|) = M$.
\end{definition}

It is a result of Whitehead \cite{Whitehead1940} that any smooth manifold can be triangulated, which is unique up to a piecewise linear homeomorphism. The name ``Whitehead triangulation'' is sometimes used to emphasize the compatibility between the triangulation and the smooth structure of $M$.

\section{Subdivisions} \label{sec:crysSubdiv}

Various methods exist to subdivide a triangulation into smaller simplices, each with their own properties. We focus on a method called crystalline subdivision (\cref{ssec:crysSubdiv}), whose main benefit is that simplices do not get too distorted when subdividing (\cref{sec:propCrystalline}). We discuss in \cref{sec:coveringPolyhedra} how one covers a polyhedron with nice subpolyhedra; we need this to ultimately establish (in \cref{sec:TjigglingRel}) a version of jiggling that is relative in the domain.

%%%%%%%%%%%%%%%%%%%%%%%%%%%%%%%%%%%%%%%%%%%%%%%%%%%%%%%%%%%%%%%%%%%%%%%%%%%%
\subsection{Crystalline subdivision} \label{ssec:crysSubdiv}

There are various ways of defining crystalline subdivision, and here we follow the definition of Thurston from \cite[p. 227]{Th2}. Crystalline subdivision as described below is based on the observation that we know how to subdivide a cube into smaller cubes. We illustrate the procedure in \cref{fig:crystsubdivT}.

We recall that a simplex, or more generally a simplicial complex, is ordered if its set of vertices is endowed with a total order.

\begin{figure}[h]
	\includegraphics[width=0.4\textwidth,page=4,trim = 0 1.5cm 0 0,clip ]{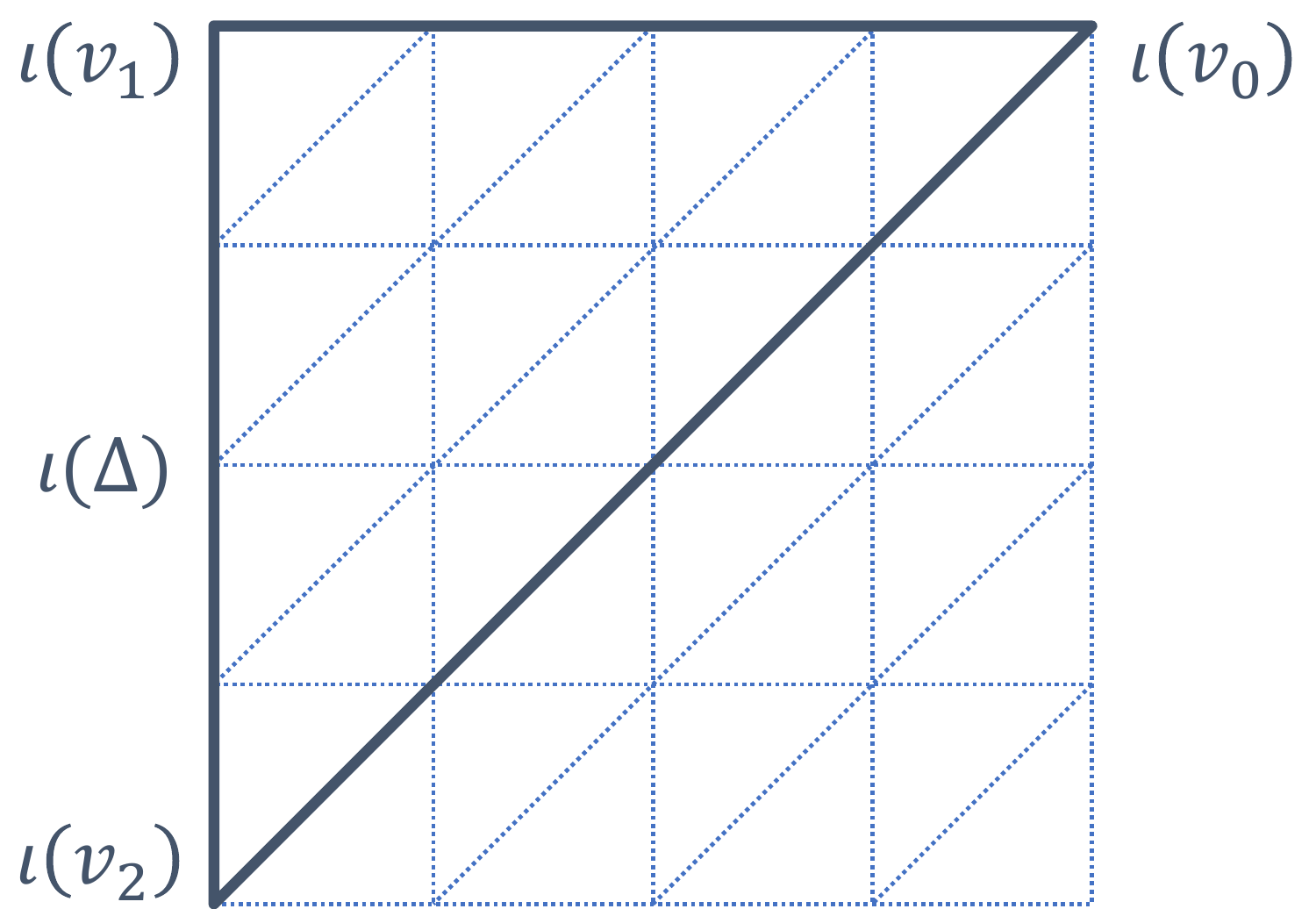}
	\centering
	\caption{The second crystalline subdivision of a $2$-simplex $\Delta$. } \label{fig:crystsubdivT}
\end{figure}

\begin{definition} \label{def:crystsubdivT}
Applying \textbf{crystalline subdivision} to an (ordered) linear $m$-simplex $\Delta$ with $m>0$ is done using the following steps:
	\begin{enumerate}
		\item \label{item:simplexToCube} Include $\Delta =\langle v_{i_0},\dots,v_{i_m} \rangle$ into the standard $m$-cube $I^m$ via the piecewise linear map $\iota$ defined by 
        \[\iota(v_{i_j}) = ({\overbrace{0,\dots,0}^{j }} ,{\overbrace{1,\dots,1}^{m-j }}).\]
		\item \label{item:subdivcube} Subdivide $I^m$ into $2^m$ smaller $m$-cubes of size $1/2$.
		\item Subdivide each of the smaller $m$-cubes into $m!$ smaller linear $m$-simplices. Each linear simplex corresponds to a permutation $\pi$ of $\{0,1,\dots,m\}$ by identifying a permutation $\pi$ with the subset $\{(x_1,\dots, x_m) \in \R^m \mid 0\leq x_{\pi(0)} \leq x_{\pi(1)} \leq \dots \leq x_{\pi(m)} \leq 1\}$. 
        \item Use $\iota$ to pullback the above subdivision of $\iota(\Delta)$ to $\Delta$.
	\end{enumerate}
We define the crystalline subdivision of $\Delta^0$ as itself.
\end{definition}
This procedure makes use of the ordering of the vertices in step~\ref{item:simplexToCube}. Indeed, for $m\geq3$, the resulting subdivision would be different if we applied an affine isomorphism that permutes the vertices of the simplex $\Delta$. It is possible to assign an ordering to the result of crystalline subdivision of $\Delta$, but as we have no use for this, we do not. 

Crystalline subdivision can be iterated but it turns out to be easier to define finer crystalline subdivisions directly. Hence we define the \textbf{$\ell$th crystalline subdivision} of the standard $m$-simplex $\Delta^m$ by instead subdividing the $m$-cube into $2^{\ell m}$ smaller $m$-cubes of size $2^{-\ell}$ in \cref{item:subdivcube} from \cref{def:crystsubdivT}.

To generalize to a simplicial complex $K$, we choose an ordering on $K$ such that the resulting subdivision is well-defined.
\begin{definition}
Let $K$ be an ordered simplicial complex. Its \textbf{$\ell$th crystalline subdivision} $K_\ell$ is defined as the union of the $\ell$th crystalline subdivisions of its simplices.
\end{definition}
Given a triangulation $T: |K| \rightarrow M$ associated to a specified ordered simplicial complex $K$, we can also speak of its $\ell$th crystalline subdivision $T_\ell$, seen as the collection of embeddings $(T_\Delta : \Delta \to M)_{\Delta \in K_\ell}$ given by restricting $T$.

%%%%%%%%%%%%%%%%%%%%%%%%%%%%%%%%%%%%%%%%%%%%%%%%%%%%%%%%%%%%%%%
\subsection{Properties of crystalline subdivision} \label{sec:propCrystalline}

Next we discuss some of the properties of crystalline subdivision. They all have the same flavor: since crystalline subdivision does not distort simplices, we are able to obtain various quantitative bounds independent of the order of subdivision $\ell$.

\subsubsection{Vertex-link}
The first advantage is that we are able to bound the maximum size of the vertex-link of each vertex in a simplicial complex and all its crystalline subdivisions. We first prove this in the case where the simplicial complex consists of a single top-dimensional simplex, after which we generalize to finite simplicial complexes.
\begin{lemma} \label{lem:crysboundlink}
	Let $\Delta^m$ be the standard $m$-simplex. Then there exists $C \in \N$ such that for all $\ell\in\N$ and all vertices $v \in \Delta^m_\ell$ we have $| \vlink(v) | \leq C$.
\end{lemma}
\begin{proof}
	Let $Q$ be the unit cube of dimension $m$ subdivided into $m!$ simplices of dimension $m$ as in \Cref{def:crystsubdivT}. Let $R$ be the same cube but now subdivided into $2^{\ell m}$ cubes of dimension $m$ and of size $2^{-\ell}$. Then from \Cref{def:crystsubdivT} we see that the size of the vertex-link of a vertex $v$ in $R$ can be bounded by the number of $m$-cubes in $R$ containing $v$, times the size of the vertex-link of a vertex when an $m$-cube is subdivided into $m$-simplices. The latter we can easily bound by the number of vertices in an $m$-cube (minus 1), since when subdiving a cube into simplices the number of vertices does not change. Hence, if we let $w$ denote a vertex in $Q$, we see that
	\begin{align*}
		| \vlink(v) | &\leq |\{D \in R \mid v \in D \text{, } D \text{ is an $m$-cube}\} | \cdot |\vlink_{Q}(w) | \\
		&\leq 2^m \cdot (| \{v \in I^m \mid v \text{ is a vertex} \} | -1) \\
		&\leq 2^m (2^m -1). \qedhere
	\end{align*}
\end{proof}

\begin{corollary} \label{cor:crysboundlink}
	Let $K$ be a finite simplicial complex of dimension $m$. Then there exists $C \in \N$ such that for all $\ell\in\N$ and vertices $v \in K_\ell$ we have $| \vlink(v) | \leq C$.
\end{corollary}
\begin{proof} 
	Denote by $\Delta \in K$ a simplex in $K$ which contains $v$. Then we see that the size of the vertex-link can be bounded by size of the vertex-link contained in $\Delta$, times the number of simplices in $K$ incident to $v$. Hence we see that
	\begin{align*}
		| \vlink(v) | &\leq |\vlink(v) \cap \Delta| \cdot | \{ \Delta_\sigma \in K \mid \Delta_\sigma \text{ incident to } v \} | \\
		&\leq 2^m (2^m -1) \cdot |\{ \Delta_\sigma \in K \} |	\qedhere
	\end{align*}
\end{proof}

\subsubsection{Model simplices} \label{sec:modelSimpl}
Another property of crystalline subdivision is that each of the simplices in the $\ell$th crystalline subdivision of a simplex $\Delta$ is equivalent to a simplex in the first subdivision of $\Delta$, up to scaling and translation. Model simplices appear already in Thurston's work~\cite{Th2}, albeit in a different formulation. 

\begin{lemma} \label{lem:csubdivExact}
Let $K$ be a finite, ordered simplicial complex in $\R^N$ of pure dimension. Then, there exists a finite collection of \textbf{model simplices} $\SC = \{\Delta_i \subset \R^N \mid i = 0,\dots, I\}$ with the following property: for any $\ell \in \N$ and $\Delta \in K_\ell^{(\topd)}$ there exists 
	\begin{itemize}
		\item a model simplex $\Delta_i \in \SC$, and
		\item a map $t : \R^N \to \R^N$ which is a composition of a translation and a scaling by $2^\ell$
	\end{itemize}
such that $\Delta_i = t(\Delta)$.
\end{lemma}

\subsubsection{Shape of linear simplices} \label{sec:shapeSimplices}
In this section we introduce three quantities related to a simplex, which tell us about the size and shape of a simplex. 

To start with, given a simplex, we are interested in the maximal ($\rmax$) and minimal ($\rmin$) distances between a vertex and the face opposite to it. The former quantity controls for instance how well a map is approximated by its linearization with respect to a given triangulation (see \cref{sec:linearizing}). Their ratio on the other hand tells us how degenerate the simplex is. 
For $\rmax$ we note that the maximum distance between a vertex and its opposite face agrees with the maximum length of the edges adjacent to both the vertex and the opposite face. For $\rmin$ we interpret the distance between vertex and opposite face as the distance between the vertex and the affine plane spanned by the face. Hence we introduce the notation $\aspan(S)$ for the \textbf{affine span} of a subset $S \subset \R^n$.

Below we define $\rmax$ and $\rmin$ for $(m+1)$-tuples of points in $\R^N$, but we can also speak of $\rmax$ and $\rmin$ of a linear simplex $\Delta$ by identifying $\Delta$ with its vertices.

\begin{definition} \label{def:rminmax}
	The functions \[\rmin, \rmax : \R^N \times \dots \times \R^N = \R^{N (m+1) } \to \R\] are defined by
	\begin{align*}
		\rmin(p_0,\dots, p_m) &= \min_{i \in [m]} d \left(p_i, \aspan \left( \langle p_0, \dots, \hat{p_i},\dots,p_m \rangle \right) \right) \text{ \quad and} \\
		\rmax(p_0,\dots, p_m) &= \max_{i,j \in [m]} d(p_i,p_j).
	\end{align*}
\end{definition}

Another quantity describing the degeneracy of a simplex is the following quantity, which we denote by $\maxcoeff$. 
As in the case of $\rmax$ and $\rmin$, we define $\maxcoeff$ for $(m+1)$-tuples of points in $\R^N$. The difference is however that here we need to assume that the simplex $\Delta$ is ordered to be able to speak of $\maxcoeff(\Delta)$, since $\maxcoeff$ is only invariant under permutations of its input fixing the first element. 
\begin{definition}\label{def:maxcoeff}
	The function \[\Lambda : \left\{(v_0,\dots,v_m) \in \R^N \times \dots \times \R^N = \R^{N (m+1) } \mid v_i \neq v_j \text{ for all } i \neq j \right\}  \to \R\] is defined by
	\[ \Lambda(v_0,\dots,v_m) = \max_{\substack{  \lambda_1,\dots,\lambda_m \in \R \\ |\sum \lambda_i (v_i-v_0)| =1} }  |\lambda_i|.\]
\end{definition}
We point out that if $v_0=0$ and the points $v_1,\dots,v_m$ form an orthogonal frame, the quantity $\Lambda(v_0,\dots,v_m)$ equals $1/\rmin(v_0,\dots,v_m)$.\footnote{For the interested reader, we note that a closed form expression can be obtained for $\maxcoeff$ by using the Lagrange multiplier method. If $m=2$ and $v_0=0$, we obtain for instance that
\[\Lambda(0,v_1,v_2) = \frac{\max\{|v_1|,|v_2|\}}{\sqrt{|v_1|^2|v_2|^2-\langle v_1,v_2 \rangle^2}}.\] In this case we see that the denominator resembles the Cauchy-Schwarz inequality and hence measures the degeneracy of the simplex $\langle 0, v_1,v_2\rangle$. }

A priori we cannot bound the quantities $\rmin$, $\rmax$ and $\maxcoeff$ among all subdivisions of a linear polyhedron, since the space of all linear simplices in $\R^N$ is not compact. However, due to the existence of the model simplices (from \cref{lem:csubdivExact}), we can bound these quantities when considering crystalline subdivisions. We point out that the product of $\rmax$ and $\maxcoeff$ is in particular independent of the number of subdivisions. Additionally, we observe that $\rmin$ and $\rmax$ decrease when we apply crystalline subdivisions, whereas $\Lambda$ increases.
\begin{lemma} \label{lem:rmaxminBound}
	Let $K$ be simplicial complex that is ordered and finite. Then, there exists $B,C,D,E \in \R_+$ such that for all $\ell \in \N$:
	\begin{align*}
		\min_{\Delta \in K_\ell} \rmin(\Delta) = B \cdot 2^{-\ell}&, \quad \max_{\Delta \in K_\ell} \rmax(\Delta) = C \cdot 2^{-\ell}     \\   
		\max_{\Delta \in K_\ell} \Lambda(\Delta) = D \cdot 2^{\ell}   \quad \textrm{ and }& \quad    \max_{\Delta \in K_\ell} \rmax(\Delta) \cdot  \maxcoeff(\Delta) = CD .
	\end{align*}
\end{lemma}

%%%%%%%%%%%%%%%%%%%%%%%%%%%%%%%%%%%%%%%%%%%
\subsection{Nice covers of polyhedra} \label{sec:coveringPolyhedra}

Our arguments often have to be localized to subpolyhedra. To this end, it is important for us to be able to cover a given polyhedron by subpolyhedra that are nice. We explain how to do this now.

Before we get to the key definition, recall that the join between two subsets $A,B$ of $\R^N$ is the set
\[ \join{A}{B} = \{ t a + (1-t) b \mid a \in A , b \in B \text{ and } t \in [0,1]\}. \]
In particular, when we take the join of a (suitable) pair of linear simplices $\Delta_1,\Delta_2$ we end up with a higher-dimensional simplex $\Delta \coloneq \join{\Delta_1}{\Delta_2}$ having the two original simplices as opposing faces.

\begin{definition} \label{def:niceSubcomplex}
Let $K$ be a simplicial complex and let $K'$ be a subcomplex. We will say that $K'$ is \textbf{nice} if for each simplex $\Delta \in \str(K')$ the subcomplex $\Delta \cap K'$ is a face of $\Delta$.
\end{definition}

\begin{figure}[h]
    \centering
    \begin{subfigure}[t]{0.49\linewidth}
        \centering
        \includegraphics[width=\linewidth,page=9, clip=true, trim = 0 3cm 0 0cm]{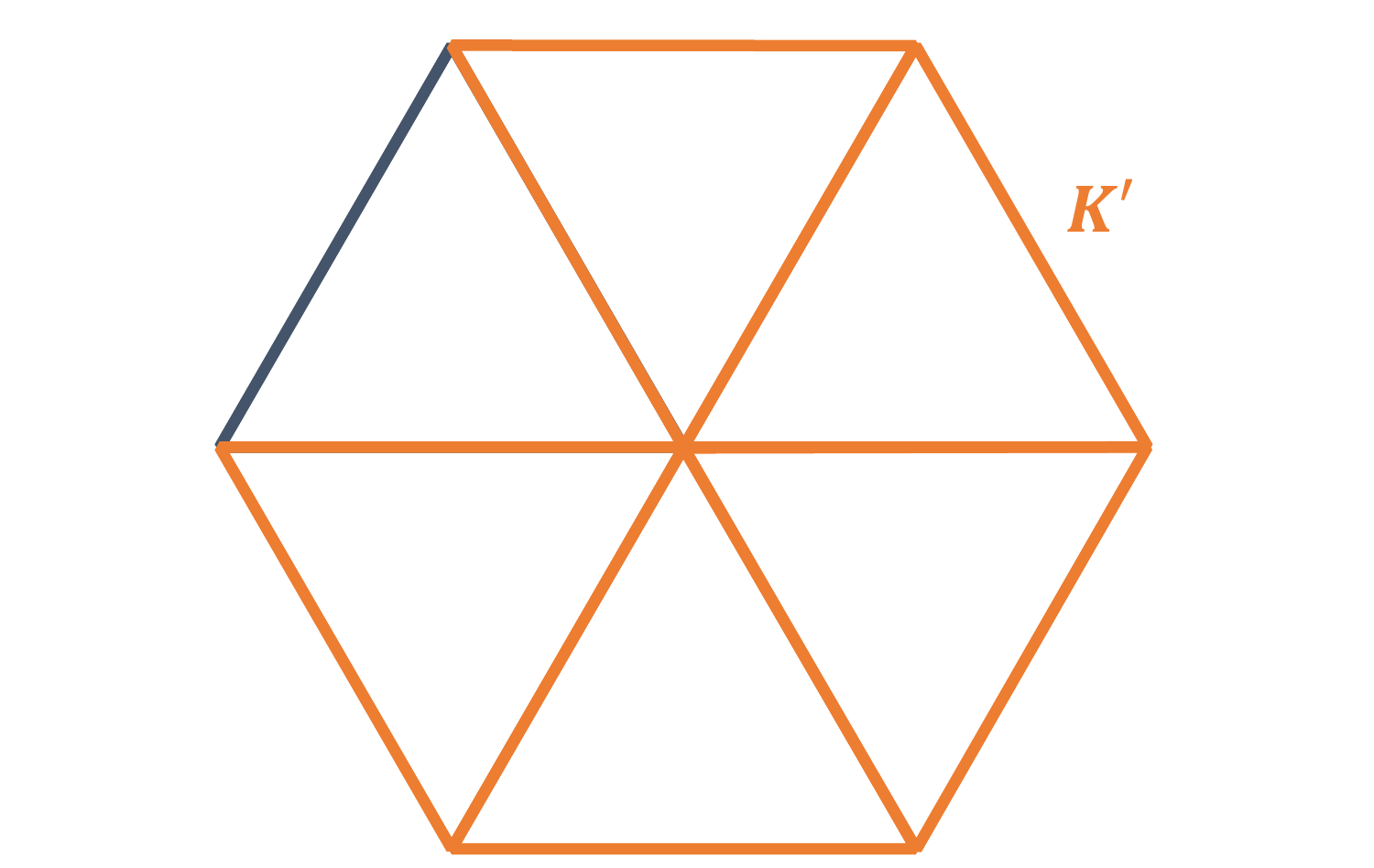}
    \end{subfigure}
    \hfill
    \begin{subfigure}[t]{0.49\linewidth}
        \centering
        \includegraphics[width=\linewidth,page=10, clip=true, trim = 0 3cm 0 0cm]{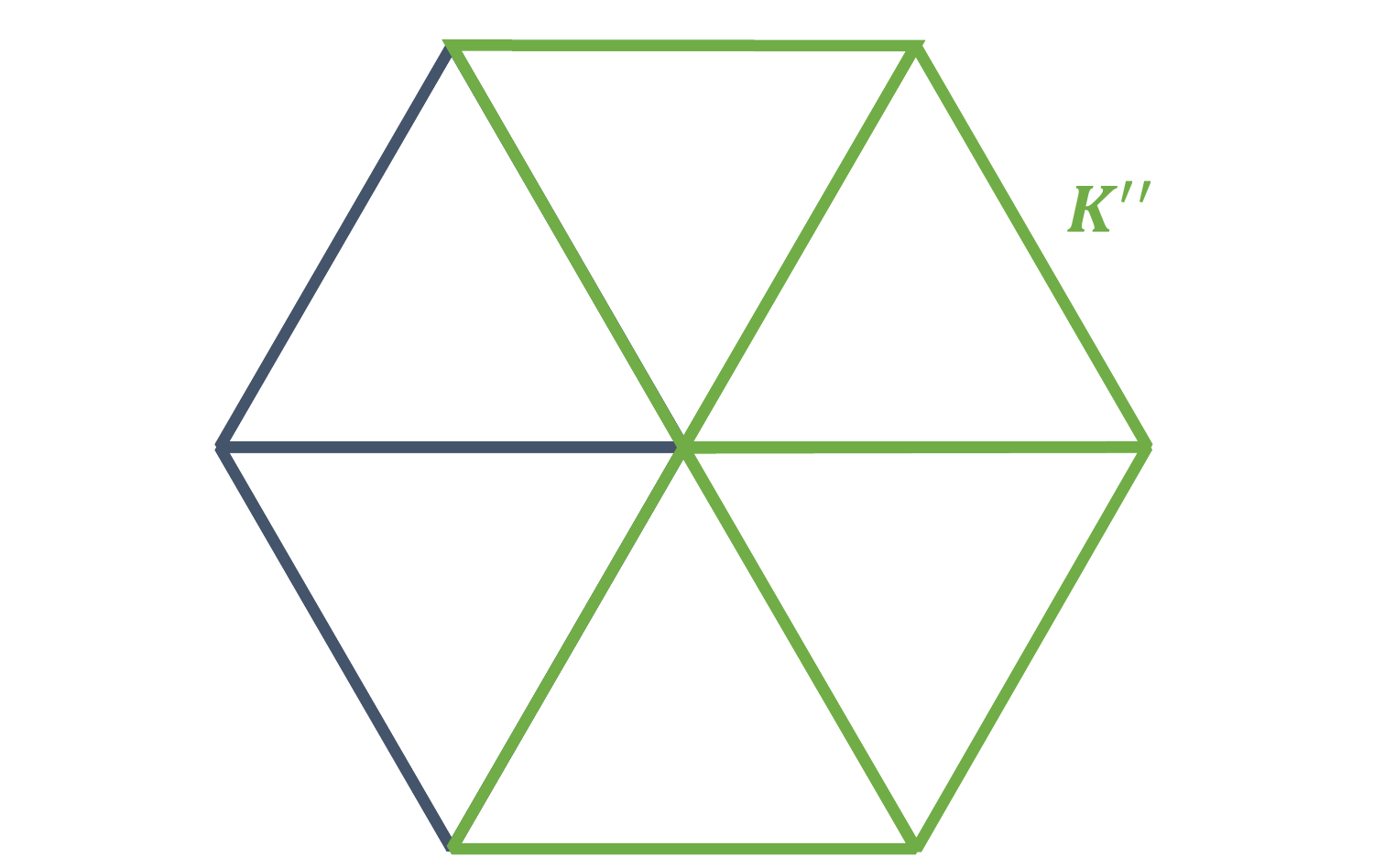}
    \end{subfigure}
	\centering
	\caption{We illustrate here subcomplexes $K'$ and $K''$ of a simplicial complex $K$, where $K'$ is not nice and $K''$ is.}  \label{fig:niceSubcomplex}
\end{figure}

We have illustrated \cref{def:niceSubcomplex} in \cref{fig:niceSubcomplex}. The meaning of niceness is that any $\Delta$ in the ring of $K'$ can be thus seen as the join of two faces, $A$ and $B$, with $A \in K'$ and $B$ disjoint from $K'$. This is useful in order to interpolate from a map/section given over $A$ to a map/section given over $B$ (as in \cref{ssec:relativeLinearization}).

The question now is how to cover a simplicial complex by nice subcomplexes. One option is to use barycentric subdivision, since every simplex in a barycentric subdivision intersects the boundary of an unsubdivided simplex in at most one face. Hence we obtain the following.

\begin{lemma}\label{lem:niceBarySubdiv}
	Let $K$ be a simplicial complex with a subcomplex $K'$. After barycentrically subdividing $K$ once, the subcomplex subdividing $K'$ is nice.
\end{lemma}

\section{Linearization of piecewise maps} \label{sec:linearizing}
In this section we focus on the the linearization of maps, which allows us to (locally) reduce the study of piecewise smooth maps to piecewise \emph{linear} maps:
\begin{definition} \label{def:linearization}
Let $K$ be a simplicial complex and let $|K|$ be the corresponding linear polyhedron. Let $s: |K| \to \R^n$ be a map that is piecewise smooth with respect to $K$. The \textbf{linearization} $s^\lin : |K| \to \R^n$ of $s$ with respect to $K$ is the unique piecewise linear map such that $s^\lin$ and $s$ agree on the vertices of $K$.
\end{definition}
We observe that $s$ and $s^\lin$ are homotopic through piecewise smooth maps, thanks to linear interpolation.

%%%%%%%%%%%%%%%%%%%%%%%%%%%%%%%%%%%%%%%%%%%
\subsection{The linearization statement}

Using Taylor's approximation theorem, we obtain the following result, stating that a (local) linearization $s_\ell'$ of a function $s$ is an increasingly better approximation of the function $s$ if we apply crystalline subdivision to the underlying simplicial complex. 

\begin{proposition}\label{cor:linSecDist}
	Let $K$ be a simplicial complex in $\R^N$. Let $s: |K| \to \R^n$ be a map that is piecewise smooth with respect to $K$. Let $K' \subset K$ be a finite subcomplex and consider a neighborhood $\Op(|K'|)$.
	
	Then, for each sufficiently large $\ell$ there exists a map $s_\ell': |K| \to \R^n$ that is piecewise smooth with respect to $K_\ell$ and moreover:
	\begin{itemize}
		\item $s_\ell'$ is piecewise linear over $|K'|$,
		\item $s_\ell' = s$ outside of $\Op(|K'|)$,
		\item $\dist{0}(s_\ell',s) = O(2^{-2\ell})$, and
		\item $\dist{1}(s_\ell',s) = O(2^{-\ell})$.
	\end{itemize}
\end{proposition}

We recall from \cref{sec:PLSmaps} that the $C^0$ and $C^1$-distances used here are independent of $\ell$; they are computed using a fixed collection of metrics on jet spaces over the simplices of $K$. If in the above setting $K'=K$, we introduce the notation $s^\lin_\ell$ for $s'_\ell$.

Since being a piecewise embedding is an open condition for the $C^1$-topology, the above result implies in particular the following:

\begin{corollary} \label{lem:linearizationEmbedding}
	Let $f: |K| \to \R^n$ be a piecewise embedding of a finite simplicial complex $K$. For $\ell$ large enough, the map $f^\lin_\ell$ is also a piecewise embedding.
\end{corollary}

%%%%%%%%%%%%%%%%%%%%%%%%%%%%%%%%%%%%%%%%%%%
\subsection{Modifying maps on a face} \label{ssec:relativeLinearization}

As explained above, we linearize in order to work locally in a linear manner. This means that we need to explain how our local arguments globalize. To this end, we introduce an interpolation procedure between sections.
\begin{definition} \label{def:interpol}
Let $\Delta$ be a linear simplex given as the join of two opposing faces $A$ and $B$. Let $t: \Delta \to [0,1]$ be the join parameter: the affine function that is $0$ over $A$ and $1$ over $B$. Suppose that $s_A,s_B : \Delta \to \R^n$ are two smooth maps. We define the \textbf{interpolation} of $s_A$ and $s_B$ over $\Delta$ with respect to $A$ and $B$, denoted as $\interpolate{\Delta}{A,B}{s_A}{s_B} : \Delta \to \R^n$, to be the map $x \mapsto t(x) s_A(x) + (1-t(x)) s_B(x)$.
\end{definition}

The result of interpolation between sections is controlled by the shape of the simplex and the difference between the two sections that serve as input:
\begin{lemma}\label{lem:interpol}
Fix a linear $m$-simplex $\Delta$ spanned by two opposing faces $A, B$. Consider moreover three maps $u_1,u_2: A \rightarrow \R^n$, $v: B \to \R^n$, and the corresponding interpolations $s_i = \interpolate{\Delta}{A,B}{u_i}{v}$. Then the following bounds hold:
\begin{itemize}
	\item $\dist{0}(s_1,s_2) = O(\dist0 (u_1,u_2))$, and
	\item $\dist{1}(s_1,s_2) = \dist{0}(s_1,s_2) + O\left(\frac{\dist0 (u_1,u_2)}{\rmin(\Delta)} + \dist1 (u_1,u_2) \right)$.
\end{itemize}
\end{lemma}
Here we can use the usual Euclidean $C^0$ and $C^1$-distances, although any other choices are equivalent up to a constant.

\subsubsection{The linear case}
We note that when we interpolate sections that are linear over a simplex $\Delta$, the result is in general not linear. Hence we define the following variation of \cref{def:interpol}, which only depends on the sections restricted to the faces.

\begin{definition} %\label{def:join}
    Let $\Delta$ be a linear simplex given as the join of two opposing faces $A$ and $B$. Suppose that $s_A: A \to \R^n$ and $s_B : B \to \R^n$ are two affine maps. We define the \textbf{join} of $s_A$ and $s_B$ over $\Delta$, denoted as $\joins{\Delta}{A,B}{s_A}{s_B}$, to be the unique affine map $\Delta \to \R^n$ agreeing with $s_A$ over $A$ and agreeing with $s_B$ over $B$.
\end{definition}
For the join we now obtain slightly simpler bounds when perturbing, than for the interpolation in \cref{lem:interpol}: 
\begin{lemma}\label{lem:join}
Fix a linear $m$-simplex $\Delta$ spanned by two opposing faces $A, B$. Consider moreover the linear maps $v: A \to \R^n$ and $u_1,u_2: B \rightarrow \R^n$, and the corresponding joins $s_i = \joins{\Delta}{A,B}{v}{u_i}$. Then the following bounds hold:
\begin{itemize}
	\item $\dist{0}(s_1,s_2) = O(\dist0 (u_1,u_2))$, and
	\item $\dist{1}(s_1,s_2) = \dist{0}(s_1,s_2) + O \left( \dist0 (u_1,u_2) \cdot \maxcoeff(\Delta) \right)$.
\end{itemize}
\end{lemma}

The above provides us in particular with bounds on the distance between a piecewise linear map, and the result of perturbing that map in each vertex.
\begin{lemma} \label{lem:perturbVert}
	Let $f: |K| \to \R^n$ be a piecewise linear map with respect to a finite simplicial complex $K$ and write $\maxcoeff=\max_{\Delta \in K}\maxcoeff(\Delta)$. Let $f': |K|\to\R^n$ be a piecewise linear map, obtained by perturbing $f$ in each vertex by at most $\epsilon>0$. Then the following bounds hold:
	\begin{itemize}
		\item $\dist0 (f,f') = O(\epsilon)$, and
		\item $\dist1 (f,f') = O(\epsilon(1+\maxcoeff))$.
	\end{itemize}
\end{lemma}
\begin{proof}
	Let $\Delta\in K$ be a top-dimensional simplex and write it as the join of two faces $A$ and $B$. Then we first perturb $f$ in the vertices of $A$ by at most $\epsilon$, which yields a function $f''$. We bound the difference between $f$ and $f''$ by applying \cref{lem:join}. Next we perturb $f''$ in the vertices of $B$ by at most $\epsilon$, after which the claim follows from again applying \cref{lem:join}.
\end{proof}

When the piecewise linear map we consider is the linearization of a map, we obtain the following. 
\begin{corollary} \label{cor:perturbLin}
	Let $f: |K| \to \R^n$ be a piecewise embedding of a finite simplicial complex $K$. There exists $\epsilon>0$, such that for all $\ell$ large enough, if we perturb all vertices of $(f^\lin_\ell,K_\ell)$ by at most $\epsilon/2^\ell$, the resulting piecewise linear map $f'_\ell$ is a piecewise embedding and satisfies 
	\begin{itemize}
		\item $\dist0 (f,f'_\ell) = O(2^{-2\ell}+\epsilon/2^\ell)$, and
		\item $\dist1 (f,f'_\ell) = O(2^{-\ell}+\epsilon)$.
	\end{itemize} 
\end{corollary}
\begin{proof}	
	We first establish the bounds on the distances. We bound the $C^0$ and $C^1$-distances between $f$ and $f^\lin_\ell$ as respectively $O(2^{-2\ell})$ and $O(2^{-\ell})$ using \cref{cor:linSecDist}. Next, we bound the $C^0$ and $C^1$-distances between $f^\lin_\ell$ and $f'_\ell$ as $O(\epsilon/2^\ell)$ and $O(\epsilon)$ using \cref{lem:perturbVert,lem:rmaxminBound}, where we recall that \cref{lem:rmaxminBound} bounds $\max_{\Delta \in K}\maxcoeff(\Delta)$. This provides us with the bounds from the statement. 
	
	If we now choose $\epsilon$ small enough then for all $\ell$ large enough, we obtain, as in \cref{lem:linearizationEmbedding}, that $f'_\ell$ is a piecewise embedding.
\end{proof}

%%%%%%%%%%%%%%%%%%%%%%%%%%%%%%%%%%%%%%%%%%%

\section{Transversality} \label{sec:prelimTransv}
In this section we discuss various notions of transversality with respect to a distribution. Here we understand by a distribution on a manifold $N$ a subbundle of $TN$ of constant rank. In \cref{sec:transSimpl} we first consider simplices. Our convention is to only consider non-degenerate $d$-simplices. We generalize to piecewise smooth embeddings in \cref{sec:transMaps}. In \cref{sec:Gr} we state some foundational facts about the Grassmannian.

\subsection{Transversality of simplices}\label{sec:transSimpl}
We introduce the notation \[+: \Gr(n,\ell_1) \times \Gr(n,\ell_2) \to \sqcup_{i=\max\{\ell_1,\ell_2\}}^{n} \Gr(n,i)\] for the operation which associates to an $\ell_1$-plane and $\ell_2$-plane their span. Using this, we now define when two planes are transverse.
\begin{definition}\label{def:TransPlanes}
	Two planes $V \in \Gr(n,v)$ and $W \in \Gr(n,w)$ are \textbf{transverse} if 
	\[\dim (V + W) = \min \{v+w,n\}, \]
	which we denote by $V \trans W$.
\end{definition}

When we identify the planes $V$ and $W$ with two affine isomorphisms $\R^v \to V$ and $\R^w \to W$, we see that the above definition of transversality differs from the notion of transversality of maps. In the latter case we would require that $\dim (V + W) = n$. 
The reason for this difference lies in the observation that affine subspaces, when their dimensions are small enough, will generically avoid one another. The same holds true for submanifolds. In contrast, linear subspaces always intersect each other, just as is the case for distributions and submanifolds. Hence we can only ask for this intersection to be minimal, which is what we implement in \cref{def:TransPlanes}. We will always use this notion of transversality when referring to the transversality of planes.

We first define transversality in the case where the simplex is linear and the distribution is a constant foliation by reducing to \cref{def:TransPlanes}.
\begin{notation}
	To a linear $d$-simplex $\Delta$ in $\R^n$ we associate a point $\Gr(\Delta)$ in the Grassmannian $\Gr(n,d)$. Given $V \in \Gr(n,k)$, we also introduce the notation $\folconst{V}$ for the constant foliation on $\R^n$ whose leaves are just translated copies of $V$. 
\end{notation}
\begin{definition}
	A linear $d$-simplex $\Delta$ in $\R^n$ is \textbf{transverse} to a constant foliation $\folconst{V}$ on $\R^n$ if $\Gr(\Delta) \trans V$.
\end{definition}

To deal with $d$-simplices $\Delta$ that are not necessarily linear, and distributions $\xi$ that are not necessarily constant, we identify the tangent space $T_x \Delta$ with a point in the Grassmannian $\Gr(n,d)$.
\begin{definition}
	A simplex $\Delta$ in $\R^n$ is \textbf{transverse} to a distribution $\xi$ on $\R^n$ if for all $x\in \Delta$ we have $T_x\Delta \trans \xi_x$.
\end{definition}

\subsubsection{Quantitative transversality}\label{sec:QuantTrans}
The above definitions are only concerned with the question whether a simplex is transverse to a distribution, which is a qualitative notion. We can also wonder how to quantify the amount of transversality. A natural way of doing so is to consider how much either the simplex or the distribution, or both, can be perturbed while remaining transverse. In the context of jiggling (\cref{sec:PerturbVertex,sec:Tjiggling}), we will consider the distribution as a given, and the simplex as something to be controlled, and hence we choose to consider only perturbations of the simplex and not of the distribution in the following definition. This leads us to the following quantitative notion, which depends on a choice of metric on the relevant Grassmannian.

\begin{definition} \label{def:transPiecewise}
	Let $\epsilon>0$. A $d$-simplex $\Delta$ in $\R^n$ is \textbf{$\epsilon$-transverse} to a distribution $\xi$ on $\R^n$ if for all $x\in \Delta$ and $D \in B(\Delta_x, \epsilon) \subset \Gr(n,d)$ we have $D \trans \xi_x$.
\end{definition}

If a simplex is transverse and remains so under simultaneous $\zeta$-perturbations of each of its vertices, the simplex is $\epsilon$-transverse, which we will make precise in \cref{lem:PerturbVertToNbhGr}. Inspired by this, we introduce the following weaker notion of (quantitative) transversality in a vertex, by considering how much this single vertex can be perturbed while remaining transverse.

\begin{definition} \label{def:semitrans}
	Let $\delta>0$ be given, together with a point $p\in \R^n$ and a linear simplex $\Delta \in \R^n$. The linear simplex $\join{p}{\Delta}$ in $\R^n$ is \textbf{$\delta$-semitransverse} in $p$ to a distribution $\xi$ on $\R^n$ if all simplices $\join{p'}{\Delta}$ with $p' \in B(p,\delta)$ are non-degenerate and transverse to $\xi$.
\end{definition}

An important difference between $\epsilon$-transversality and $\delta$-semitransversality is the following: the $\epsilon$-transversality of a simplex $\Delta$ with respect to a constant foliation $\folconst{V}$ only depends on $\Gr(\Delta)$. Hence, the simplex $\Delta$ remains $\epsilon$-transverse when we perturb its vertices within $\aspan(\Delta)$ as long as we stay non-degenerate. Additionally, we can scale the simplex $\Delta$ and the result will also be $\epsilon$-transverse. This is not the case for the $\delta$-semitransversality of $\Delta$: if we scale $\Delta$ by a factor $L$, the resulting simplex will be $L\delta$-semitransverse.

A related observation is that $\epsilon$-transversality does not control the degeneracy of the simplex, whereas this is the case for $\delta$-semitransversality. Specifically, the latter tells us that the vertex is further than $\delta$ from the opposite face. Hence when dealing with $\delta$-semitransversality (as in for instance \cref{lem:semitransDiffSimplex,lem:semitransDiffFol}), we need to take the ``shape'' or ``size'' of the simplex into account.

\subsubsection{Qualitative transversality}
We now consider two stronger notions of transversality, in the sense that they also take the subsimplices of a simplex into account: stratified transversality and being in general position. For the former, we interpret a simplex as a space stratified by its subsimplices, which gives rise to the following.
\begin{definition}
	A simplex $\Delta$ in $\R^n$ is \textbf{stratified transverse} to a distribution $\xi$ on $\R^n$ if each of its subsimplices is transverse to $\xi$.
\end{definition}

The notion of general position is even stronger: we compare the directions of the subsimplices and distribution in different points in $\R^n$. In particular, this puts a restriction on how much the distribution can vary within a simplex.
\begin{definition} \label{def:ThGenPos}
	A simplex $\Delta$ in $\R^n$ is \textbf{in general position} with respect to a distribution $\xi$ on $\R^n$ if for each subsimplex $\Delta'$ of $\Delta$ (including $\Delta$ itself) and for all $x\in\Delta$ we have that $\Delta'$ is transverse to $\folconst{\xi_x}$.
\end{definition}

\subsubsection{Transversality for simplicial complexes}
The generalization from the above notions of transversality for simplices to simplicial complexes is straightforward, by requiring that every top-dimensional simplex in the simplicial complex satisfies the notion. We recall that we introduced the notation $K^{(\topd)}$ for the set of top-dimensional simplices of a simplicial complex $K$. 
\begin{definition} \label{def:transSimplComplex}
	Let $K$ be a simplicial complex, then $K$ is \textbf{($\epsilon$-/stratified) transverse} or \textbf{in general position} if each $\Delta \in K^{(\topd)}$ is ($\epsilon$-/stratified) transverse or in general position.
\end{definition}
Semitransversality can be generalized to simplicial complexes by specifying for each simplex the corresponding vertex which can be perturbed. We leave this to the reader, but note that because of this dependence on the vertex it does not make sense to consider whether semitransversality is preserved under subdivision. For the four notions of transversality in \cref{def:transSimplComplex} this does make sense, but not all are preserved under subdivision of the simplex. From their definition, we observe the following:
\begin{lemma}
	If a simplicial complex $K$ is ($\epsilon$-)transverse, then so are its subdivisions.
\end{lemma}
On the other hand, if a simplicial complex is stratified transverse or in general position, this is generally not preserved under subdivision. We give an example in \cref{fig:transvNotSubdiv}.

\begin{figure}[h]
	\includegraphics[width=0.4\textwidth,page=9]{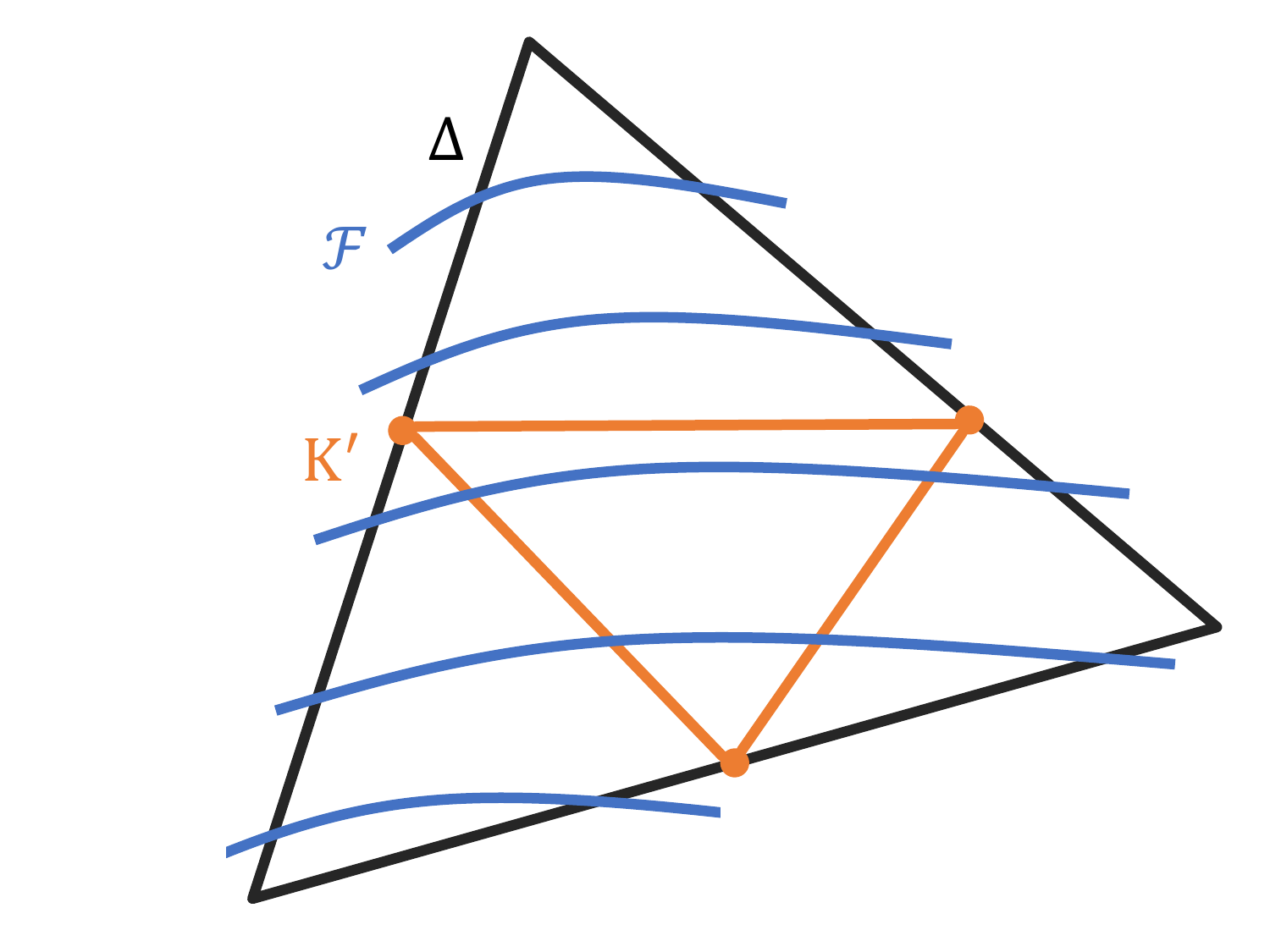}
	\centering
	\caption{An example of a linear $2$-simplex $\Delta$ in $\R^2$ which is (stratified) transverse and in general position with respect to a distribution $\fol$, but whose subdivision $K'$ is only transverse.} \label{fig:transvNotSubdiv}
\end{figure}

\subsection{Transversality of piecewise smooth maps}\label{sec:transMaps}
We now generalize all of the above notions to maps $f: M \to N$ between manifolds that are piecewise smooth embeddings with respect to a triangulation $T:|K| \to M$, and distributions $\xi$ on $N$. The idea is to consider the images of the simplices of $K$ under $f\circ T$ and determine whether these are transverse to $\xi$. Since we assume $f$ and $T$ to be piecewise smooth embeddings, these images are simplices themselves. Hence we will refer to them as the simplices of $(f,K)$ or $(f,T)$, depending on whether $f$ is smooth with respect to a simplicial complex $K$ or a triangulation $T$. We then check the transversality of a simplex of $(f,K)$ (or $(f,T)$) to $\xi$ in a point $x\in f(M)$ by identifying $T_{x}N$ with $\R^n$, where $n = \dim(N)$, and using \cref{def:TransPlanes}.
\begin{definition}
	Let $f:M\to N$ be a piecewise smooth embedding with respect to a triangulation $T: |K| \to M$ and let $\xi$ be a distribution on $N$. Then $f$ is \textbf{transverse} to $\xi$ if all top-dimensional simplices $\Delta$ of $(f,K)$ are transverse to $\xi$.
\end{definition}

Similarly, we check whether the simplices of $(f,K)$ (or $(f,T)$) are $\epsilon$- or stratified transverse by identifying with Euclidean space. We observe that $\epsilon$-transversality depends on a choice of metrics on each (Grassmannian of) $T_xN$. 

\begin{definition}
	Let $f:M\to N$ be a piecewise smooth embedding with respect to a triangulation $T: |K| \to M$ and let $\xi$ be a distribution on $N$. Then $f$ is \textbf{$\epsilon$-/stratified transverse} to $\xi$ if all top-dimensional simplices $\Delta$ of $(f,K)$ are $\epsilon$-/stratified transverse to $\xi$.
\end{definition}
In the case of stratified transversality, we also write that $(f,K)$ or $(f,T)$ is stratified transverse, to emphasize the underlying simplicial complex or triangulation of which we consider the simplices.

Being in general position is, however, not a local notion and hence we use the composition $f\circ T$ to pull back the distribution to Euclidean space.
\begin{definition} 
	Let $f: M \to N$ be a map that is piecewise smooth with respect to a triangulation $T: |K| \rightarrow M$ and let $\xi$ be a distribution on $N$. Then $(f,K)$ is \textbf{in general position} with respect to $\xi$ if:
	\begin{itemize}
		\item $f$ is transverse to $\xi$, and
		\item each top-dimensional simplex $\Delta \in K$ is in general position with respect to the distribution $(f \circ T)^*\xi$.
	\end{itemize}   
\end{definition}

\begin{remark}
	Since for semitransversality, we need to specify the vertex of a simplex in which it is semitransverse, we do not write that $(f,K)$ or $(f,T)$ is $\delta$-semitransverse. Instead, we refer to the semitransversality in $v_\Delta$ of the simplex $\Delta$ in $(f,K)$ or $(f,T)$ directly.
\end{remark}

\subsection{Grassmannian} \label{sec:Gr}
We end this section with some fundamental facts about the Grassmannian, since we will need these in \cref{sec:transEstimates}.

\subsubsection{Charts} \label{sec:GrCharts}
We recall that a chart $\phi: \SU_V \to \R^{d(n-d)}$ around a plane $V \in \Gr(n,k)$ can be constructed as follows. We let $\SU_V$ be such that for all $W \in \SU_V$, we can identify $W$ with the plane $\{x + T_W x \mid x\in V\} \subset \R^n$ where $T_W$ is a linear map $V \to V^\perp$. By choosing orthonormal bases of $V$ and $V^\perp$, we obtain a matrix $M_W$ which we interpret as an element of $\R^{k(n-k)}$. Hence the smooth chart $\phi$ is defined by \[W \mapsto M_W \in \R^{k(n-k)}.\] 

\subsubsection{Metrics}\label{sec:GrMetrics}

The Grassmannian can be endowed with various metrics. For our purposes the following metric will be convenient and will therefore be our default:
\begin{definition}\label{def:MetricGrProj}
	Let $V_1,V_2 \in \Gr(n,k)$ and denote the orthogonal projection $\R^n \to V_i \subset \R^n$ by the linear operator $P_{i}:\R^n \to \R^n$. We then define $\dproj(V_1,V_2) = \|P_1-P_2\|$, where the latter is just the operator norm for linear maps $\R^n \to \R^n$. We write $\Bproj(V,r)$ for the ball around a plane $V\in\Gr(n,k)$ of radius $r$ using the metric $\dproj$.
\end{definition}
The above metric induces the standard topology on the Grassmannian. We recall that, since the Grassmannian is compact, this metric is equivalent to any other metric, as long as it induces the standard topology.

Alternatively, we can locally define the following induced metric on the Grassmannian: We denote by $\LinMaps(V,V^\perp)$ the space of linear maps $V \to V^\perp$ endowed with the operator norm. As in \cref{sec:GrCharts}, we obtain a chart $\psi_V: \SU_V \to \LinMaps(V,V^\perp)$ which sends a plane $W \in \SU$ to $T_W \in \LinMaps(V,V^\perp)$. Pulling back the norm on $\LinMaps(V,V^\perp)$ to $\SU_V$ defines the metric $\dop{V}$ on $\SU_V$, which thus only depends on a choice of inner product on $\R^n$. 

Using the compactness of the Grassmannian we obtain the following.
\begin{lemma}\label{lem:GrMetricChartOpNorm}
	Let $n,k\in \N$ and $r\in\R_{>0}$. There exists $C_1,C_2>0$ such that for all planes $V \in \Gr(n,k)$ and $W_1,W_2\in \Bop{V}(V,r)$ the following holds
	\[ C_1 \dproj(W_1,W_2) < \dop{V}(W_1,W_2) < C_2 \dproj(W_1,W_2).\]
\end{lemma}
\begin{proof}
	Fix $V\in \Gr(n,k)$. 
	By compactness of $\overline{\Bop{V}(V,r+1)}$, the metric $\dproj$ is equivalent to $\dop{V}$ when we restrict both to $\overline{\Bop{V}(V,r+1)}$. Hence there exists $C_1$ and $C_2 $ as in the statement for $V$ specifically. Since the Grassmannian itself is compact, we can take $C_1$ and $C_2$ to be independent of $V$.
\end{proof}

\section{Transversality estimates} \label{sec:transEstimates} 

After introducing various notions of transversality in \cref{sec:prelimTransv}, we now prove some related results. We start in \cref{sec:TransCrit} by introducing some criteria to check the transversality and semitransversality of a simplex with respect to a constant foliation. Next, in \cref{sec:STToETrans} we show how to deduce $\epsilon$-transversality from $\delta$-semitransversality, under some additional assumptions. Since these results both deal with transversality with respect to a constant foliation, we generalize to arbitrary distributions in \cref{sec:nonConstantDistr}.

\subsection{Transversality via projecting} \label{sec:TransCrit}
In \cref{def:TransPlanes} we defined the transversality of two linear subspaces of $\R^n$ by requiring that the space they span is maximal. An equivalent characterization of transversality of two planes of sufficiently small dimension is the following.
\begin{lemma} \label{lem:TransPlanesQuotient}
	Let $v,w,n\in\N$ be such that $v+w \leq n$. Two linear subspaces $V$ and $W$ of respectively dimension $v$ and $w$ in $\R^n$ are transverse if and only if the (orthogonal) projection $\R^n \to \R^n / W$ is injective when restricted to $V$.
\end{lemma}
In particular, we see that a linear simplex $\Delta$ in $\R^n$ is transverse to a foliation $\fol(V)$ if and only if the map $\pi_V:\R^n \to \R^n / V \cong V^\perp$ is injective when restricted to $\Delta$.

We now use the above observation to characterize the transversality of a linear $d$-simplex in a more inductive manner. We recall the notation $\aspan(S) \subset \R^n$ for the affine span of a set $S \in \R^n$. We illustrate \cref{lem:transSimplex} and its proof in \cref{fig:transSimplex}.

\begin{lemma} \label{lem:transSimplex} 
	Fix $V \in \Gr(n,k)$. Let $p \in \R^n$ be a point and $\Delta$ a linear $(d-1)$-simplex in $\R^n\setminus \{p\}$ with $d \leq n-k$. Assume that $\Delta$ is transverse to the foliation $\folconst{V}$. Then, the linear simplex $\join{p}{\Delta}$ is transverse to $\folconst{V}$ if and only if $\pi_V(p) \notin \pi_V(\aspan(\Delta))$.
\end{lemma}

\begin{figure}[h]
	\centering
	\begin{subfigure}[b]{0.4\textwidth}
		\centering
		\includegraphics[width=\textwidth,page=2]{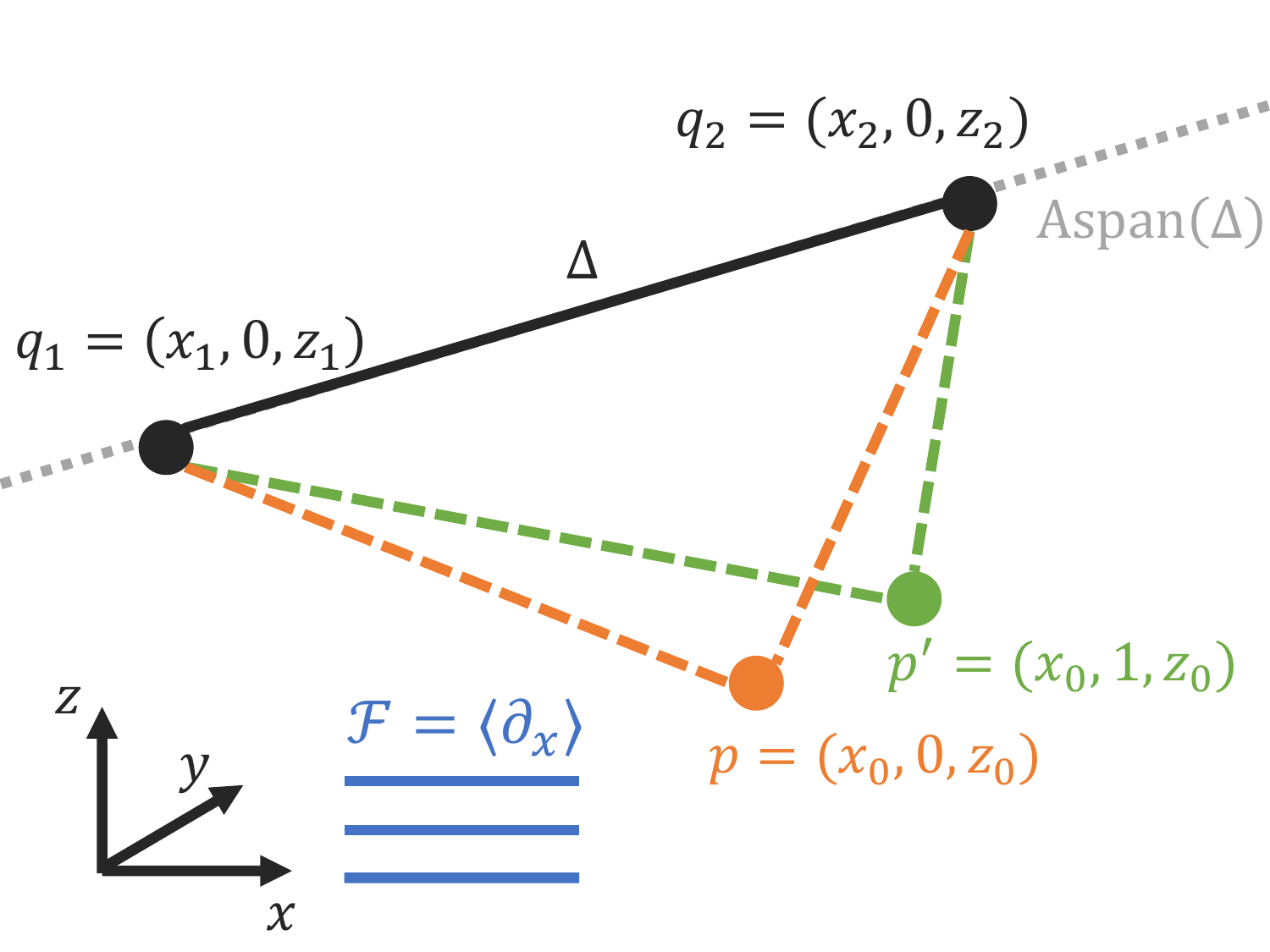}
	\end{subfigure}
	\raisebox{2cm}{\begin{tikzpicture}
			\draw[thick,->] (0,0) -- ++ (0.05\textwidth,0);
			\node at (0.025\textwidth,0.3cm) {$\pi$};
	\end{tikzpicture}}
	\begin{subfigure}[b]{0.4\textwidth}
		\centering
		\includegraphics[width=\textwidth,page=3]{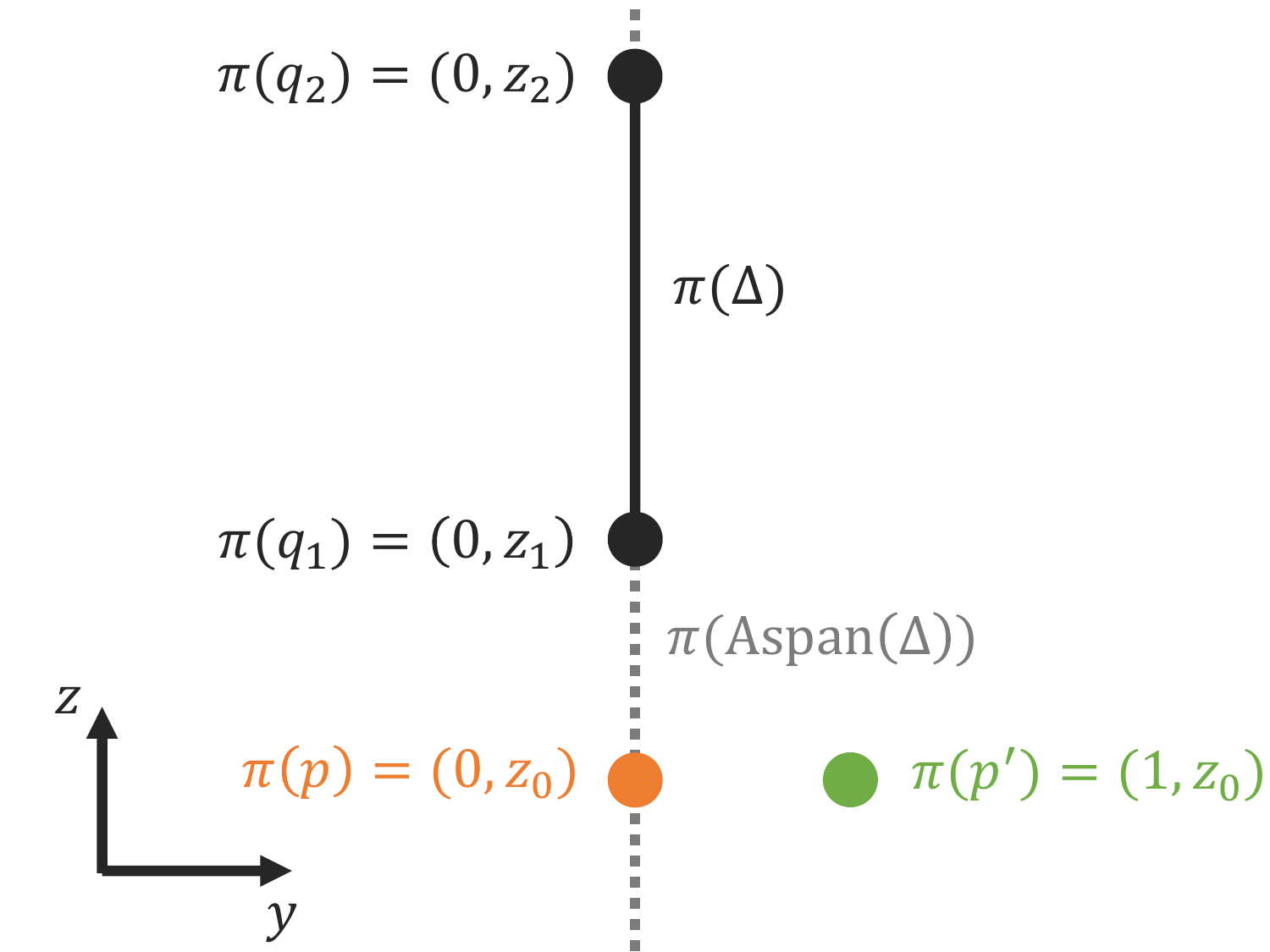}
	\end{subfigure}
	\centering
	\caption{On the left we see a $1$-simplex $\Delta$ in $\R^3$ that is transverse to the foliation $\fol = \langle \partial_x \rangle$ together with two points $p$ and $p'$. The simplex $\join{p}{\Delta}$ is not transverse, whereas $\join{p'}{\Delta}$ is. On the right, in $\R^2$, we see that indeed $\pi(p) \in \pi(\aspan(\Delta))$, whereas $\pi(p') \notin \pi(\aspan(\Delta))$.} \label{fig:transSimplex}
\end{figure}
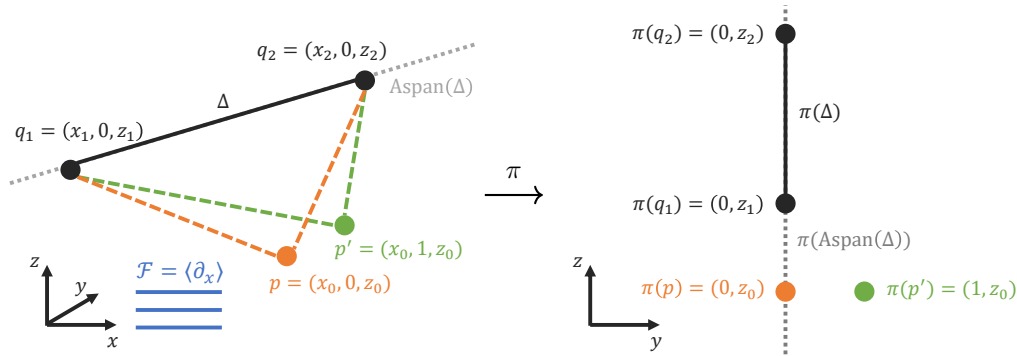
\begin{proof}	
	By assumption we know that $\Delta$ is transverse to $\folconst{V}$ and hence \cref{lem:TransPlanesQuotient} tells us that $\pi_V(\Delta)$ is a $(d-1)$-simplex. The simplex $\join{p}{\Delta}$ is transverse to $\folconst{V}$ if and only if $\pi_V(\Delta)$ is a $d$-simplex. Hence $\join{p}{\Delta}$ is transverse if and only if $\pi_V(p) \notin \pi_V(\aspan(\Delta))$.	
\end{proof}

Instead of projecting along the foliation, we can also project along the simplex. Here we write $\pi_\Delta$ for the projection $\R^n \to \R^n / \Gr(\Delta) \cong \aspan(\Delta)^\perp$.
\begin{lemma}\label{lem:StProjectSimplex}
	Fix $V\in \Gr(n,k)$. Let $p \in \R^n$ be a point and $\Delta$ a linear $(d-1)$-simplex in $\R^n\setminus \{p\}$ with $d \leq n-k$. Assume that $\Delta$ is transverse to the foliation $\folconst{V}$. Then, the linear simplex $\join{p}{\Delta}$ is transverse to $\folconst{V}$ if and only if $\pi_{\Delta}(p) \notin \pi_{\Delta}(V)$.
\end{lemma}
\begin{proof}
	We start by defining the following three projections 
	\begin{align*}
		\pi_{\join{p}{\Delta}}&: \R^n\to \R^n /\aspan(\join{p}{\Delta}) \cong \aspan(\join{p}{\Delta})^\perp \\
		\pi_{\Delta} &: \R^n\to \R^n /\aspan(\Delta) \cong \aspan(\Delta)^\perp \\	
		\pi_{\pi_\Delta(p)} &: \aspan(\Delta)^\perp \to \aspan(\Delta)^\perp/\aspan(\pi_\Delta(p)) \cong \aspan(\join{p}{\Delta})^\perp,
	\end{align*}
	where we note that the map $\pi_{\join{p}{\Delta}}$ can be written as the composition $\pi_{\pi_\Delta(p)} \circ \pi_{\Delta}$. 
	
	Since the simplex $\Delta$ is transverse to $\folconst{V}$, \cref{lem:TransPlanesQuotient} tells us that $\pi_{\Delta}(V)$ is a $k$-plane. Similarly, we know that $\join{p}{\Delta}$ is transverse to $\folconst{V}$ if and only if $\pi_{\join{p}{\Delta}}(V)$ is a $k$-plane. Hence $\join{p}{\Delta}$ is transverse to $\folconst{V}$ if and only if the $k$-plane $\pi_\Delta(V)$ is mapped to a $k$-plane under  $\pi_{\pi_\Delta(p)}$. This is the case if and only if $\pi_{\Delta}(p) \notin \pi_{\Delta}(V)$. 
\end{proof}

In \cref{lem:transSimplex,lem:StProjectSimplex} we need to assume that $d \leq n-k$ for the ``if and only if'' part of the claim. When interested in simplices of higher dimensions, we employ the following remark.
\begin{remark} \label{rem:TransDim}
	A simplex $\Delta$ of dimension $d>n-k$ is transverse to $\fol$ if and only if at least one of its $(n-k)$-simplices is transverse to $\fol$. 
\end{remark} 

\subsubsection{Semitransversality via projecting}

Using \cref{lem:transSimplex} we define the following equivalent notion of semitransversality (recall \cref{def:semitrans}) for linear simplices of dimension at most complementary to the rank of the foliation. 
\begin{corollary} \label{cor:semitransEmptyIntQuotient}
	Fix $V\in\Gr(n,k)$. Let $\delta>0$ be given, together with a point $p\in \R^n$ and a linear $d$-simplex $\Delta$ in $\R^n$ with $d <n-k$. Then, the simplex $\join{p}{\Delta}$ is $\delta$-semitransverse to $\folconst{V}$ in $p$ if and only if the simplex $\Delta$ is transverse to $\folconst{V}$ and
	\[
	B(\pi_V (p),\delta) \cap \pi_V (\aspan \Delta) = \emptyset.
	\]
\end{corollary}

Equivalently, from \cref{lem:StProjectSimplex} we obtain the following.
\begin{corollary}\label{cor:StProjectSimplex} 
	Fix $V\in \Gr(n,k)$. Let $\delta>0$ be given, together with a point $p \in \R^n$ and a linear $d$-simplex in $\R^n$ with $d<n-k$. Then, the simplex $\join{p}{\Delta}$ is $\delta$-semitransverse to $\folconst{V}$ in $p$ if and only if the simplex $\Delta$ is transverse to $\folconst{V}$ and 
	\[B(\pi_\Delta(p),\delta) \cap \pi_\Delta(V) = \emptyset . \]
\end{corollary}

\subsection{From semitransversality to \texorpdfstring{$\epsilon$}{ε}-transversality}\label{sec:STToETrans}
The goal of this section is to prove \cref{cor:inducSemiTransToTrans}, which states that if each subsimplex of a simplex $\Delta$ is $\delta$-semitransverse in one of its vertices, the simplex $\Delta$ is $\epsilon$-transverse for some $\epsilon$ that can be computed from $\delta$. Our strategy for this proof is as follows: we show that in the former setting, the simplex $\Delta$ remains transverse under simultaneous $\zeta$-perturbations of all of its vertices for some $\zeta$ depending on $\delta$, which implies in turn $\epsilon$-transversality.

\subsubsection{From perturbing all vertices to \texorpdfstring{$\epsilon$}{ε}-transversality} \label{sec:PerturbVertToETrans}
We first prove the claim that simultaneously perturbing all vertices of a $d$-simplex $\Delta$ in $\R^n$ corresponds to a neighborhood of $\Gr(\Delta)$ in the Grassmannian. For this result, we introduce the set $\SD(\Delta,\zeta)\subset \Gr(n,d)$ of planes which are the affine span of a simplex obtained from $\Delta$ by simultaneously perturbing all of its vertices by at most $\zeta$. That is, writing $\Delta = \langle v_0,\dots,v_d\rangle$, we define
\[\SD(\Delta,\zeta) = \left\{ \Gr\left(\left\langle v'_0,\dots,v'_d\right\rangle\right) \mid d(v'_i,v_i)<\zeta \text{ for all } i \in\{0,\dots,d\} \right\}.\]
We require $\zeta<\rmin(\Delta)/2$ to make sure the resulting simplices do not degenerate, where we remind the reader that $\rmin$ and $\maxcoeff$ are defined in \cref{def:rminmax,def:maxcoeff} respectively. We introduce the notation $\lspan(S)$ for the \textbf{linear span} of a subset $S \subset \R^n$.
\begin{lemma} \label{lem:PerturbVertToNbhGr}
	Let $d\in\N$ be given. Then there exist $C_1,C_2>0$ such that for all linear $d$-simplices $\Delta$ in $\R^n$ and all $\zeta>0$ satisfying \[\zeta<\min\left\{\frac{\rmin(\Delta)}{2},\frac{1}{\maxcoeff(\Delta)}\right\},\]
	it holds for all $\Delta'\in \SD(\Delta,\zeta)$ that
	\[ \frac{C_1 \zeta}{\rmax(\Delta)} <  \dproj(\Delta,\Delta')  < C_2 \zeta \maxcoeff(\Delta). \]
\end{lemma}
\begin{proof} 
	Let the set $\{v_0,\dots,v_d \}$ consist of the vertices of $\Delta$, where we assume without loss of generality that $v_0=0$.
	
	We first deal with the lower bound. Let $T: \lspan(\Delta) \to \lspan(\Delta)^\perp$ be a linear map whose operator norm is at most $\zeta/\rmax(\Delta)$, then we observe that 
	\[ |Tv_i | \leq \|T\| |v_i| \leq \|T\| \cdot \rmax(\Delta) < \zeta. \]
	It hence follows by definition that the plane $\{x+Tx \mid x\in\lspan(\Delta)\}$ is contained in $\SD(\Delta,\zeta)$. By taking $r$ in \cref{lem:GrMetricChartOpNorm} to be $n$, there exists $C_1>0$, only depending on $n$ and $d$, such that all planes in a $C_1 \zeta/\rmax(\Delta)$-neighborhood of $\Gr(\Delta)$ are contained in $\SD(\Delta,\zeta)$.

	Next we deal with the upper bound. We let $\Delta' = \langle v'_0,\dots,v'_d\rangle$ be a simplex such that $ \Gr(\Delta') \in \SD(\Delta,\zeta)$ and define its translation $\Delta''= \langle 0,v'_1-v'_0,\dots,v'_d-v'_0\rangle $. It then holds that $\Gr(\Delta') = \Gr(\Delta'')$. Hence they both correspond to the plane $\{ x+ Tx \mid x \in \lspan(\Delta) \}$, where $T$ is again a linear map $\lspan(\Delta) \to \lspan(\Delta)^\perp$. Since $d(v'_i,v_i)<\zeta$, we obtain that $| Tv_i| \leq |(v'_i-v'_0) -v_i| < 2\zeta$ for all $i\in\{1,\dots,d\}$. We then observe that
	\begin{align*}
		\|T\| &= \max_{\substack{ \lambda_1,\dots,\lambda_d \in \R \\ |\sum \lambda_i v_i | =1 }} \left| T \left( \sum \lambda_i v_i \right) \right|
		\leq \max_{\substack{ \lambda_1,\dots,\lambda_d \in \R \\ |\sum \lambda_i v_i | =1 }}  \sum | \lambda_i ||Tv_i |  \\
		&< 2\zeta \cdot \max_{\substack{ \lambda_1,\dots,\lambda_d \in \R \\ |\sum \lambda_i v_i | =1 }}  \sum | \lambda_i |
		\leq 2n \zeta \maxcoeff(\Delta).
	\end{align*} 
	Hence every plane in $\SD(\Delta,\zeta)$ is of the form $\{x+Tx \mid  x\in\lspan(\Delta)\}$ where $T$ satisfies $\|T\| < 2n \zeta\maxcoeff(\Delta)$. By taking $r$ in \cref{lem:GrMetricChartOpNorm} to be $n$, there exists $C_2>0$, only depending on $n$ and $d$, such that all planes in $\SD(\Delta,\zeta)$ are contained in a $C_2  \zeta\maxcoeff(\Delta)$-neighborhood of $\Gr(\Delta)$.
\end{proof}

In the case where these $\zeta$-perturbations of $\Delta$ preserve the transversality of $\Delta$, we obtain a lower bound for the transversality of $\Delta$.  
\begin{corollary}\label{lem:semiVSQuantTrans}
	Fix $d\in \N$ and $V\in\Gr(n,k)$. There exists $C>0$ such that for all $\zeta>0$ and all linear $d$-simplices $\Delta$ in $\R^n$ the following holds. If $\Delta$ is transverse to $\fol$ under simultaneous perturbations of each of its vertices by at most $\zeta$, then 
	 $\Delta$ is $(C\zeta/\rmax(\Delta))$-transverse to $\fol$.
\end{corollary}

\subsubsection{From semitransverse to perturbing all simplices}

Given a simplex that is semitransverse to a constant foliation $\fol(V)$, we now consider the question of what amount of semitransversality is preserved when we vary a subsimplex of codimension $1$. 
Here we use the distance $\dproj$ on $\Gr(n,d)$ from \cref{def:MetricGrProj}.

\begin{lemma} \label{lem:semitransDiffSimplex}
	Fix $d\in \N$ and $V\in\Gr(n,k)$. Let $\beta,\gamma,r>0$ be given, together with a point $p$ in $\R^n$ and two linear $d$-simplices $\Delta_1$ and $\Delta_2$ in $\R^n$ such that 
	\begin{itemize}
		\item $\Delta_1,\Delta_2 \subset B(p,r)$,		
		\item $\Delta_1$ and $\Delta_2$ are transverse to $\folconst{V}$, and
		\item $\dproj(\Gr(\Delta_1),\Gr(\Delta_2)) < \beta$.
	\end{itemize}
	If the simplex $\join{p}{\Delta_1}$ is $\gamma$-semitransverse to $\folconst{V}$ in $p$, then $\join{p}{\Delta_2}$ is $(\gamma - 2 \beta r)$-semitransverse to $\folconst{V}$ in $p$. 
\end{lemma}
\begin{proof}
	We make a slightly different claim stating that if $\Delta_1$ and $\Delta_2$ additionally share a vertex, then the $\gamma$-semitransversality of $\join{p}{\Delta_1}$ implies that $\join{p}{\Delta_2}$ is $(\gamma-\beta r)$-semitransverse. Assuming that this claim holds true, we then obtain a proof for the lemma, without this additional assumption, by applying the claim twice. Hence, in the remainder of this proof we prove the claim and assume, without loss of generality, that $\Delta_1$ and $\Delta_2$ share the origin as vertex.
	
	By \cref{rem:TransDim} we can assume that $d+1 \leq n-k$. We denote by $\pi_i:\R^n \to \R^n / \Gr(\Delta_i) $ the projections and by $P_i:\R^n \to \R^n$ the orthogonal projection onto $\lspan(\Delta_i)$ for $i=1,2$. By assumption it then holds that $\|P_1-P_2\|<\beta$. Throughout the proof we identify $V$ with its associated linear subspace of $\R^n$. 
	
	Since $\join{p}{\Delta_1}$ is $\gamma$-semitransverse to $\folconst{V}$ in $p$, we have by \cref{cor:StProjectSimplex} that 
	\begin{equation} \label{eq:STDiffS}
		B(\pi_1(p),\gamma) \cap \pi_1(V) = \emptyset.
	\end{equation}
	We note that \begin{align*}
		\pi_1^{-1}(\pi_1(p)) &= \{p\}+\lspan(\Delta_1) \text{, and} \\
		\pi_1^{-1}(\pi_1(V)) &= \lspan(V \cup \Delta_1),
	\end{align*}
	where $\{p\}+\lspan(\Delta_1)$ denotes the affine plane through $p$ parallel to $\lspan(\Delta_1)$. Hence \cref{eq:STDiffS} is equivalent to \[d(\{p\}+\lspan(\Delta_1),\lspan(V \cup \Delta_1)) > \gamma, \] which is in turn equivalent to $d(\lspan(\Delta_1),\{-p\}+V) > \gamma$. Both of these distances should be interpreted as the distance between two affine subspaces of $\R^n$. We observe that
	\begin{align*}
		d(\lspan(\Delta_1),\{-p\}+V) &= \inf_{v \in V} \|(v-p)-P_1 (v-p)\| \\ 
		&\leq \inf_{v \in V} \left( \|(v-p)-P_2 (v-p)\| + \|P_2 (v-p)-P_1 (v-p)\|\right) \\ 
		&\leq d(\lspan(\Delta_2),\{-p\}+V) + \beta \inf_{v\in V} \|v-p\|,
	\end{align*}
	which implies that
	\[ d(\lspan(\Delta_2),\{-p\}+V) \geq d(\lspan(\Delta_1),\{-p\}+V) - \beta \inf_{v\in V} \|v-p\| > \gamma - \beta r. \]
	Hence the claim follows.
\end{proof}

The next lemma shows that, using semitransversality and some additional assumptions, we can deduce that the given simplex is transverse under simultaneous perturbations of each of its vertices. Its proof boils down to inductively applying \cref{lem:semitransDiffSimplex}.

\begin{lemma} \label{lem:inducSemiTransToPerturbAll}
	Fix $d\in \N$ and $V\in\Gr(n,k)$. Then there exists $C>0$ such that for all $\delta>0$ and any linear $d$-simplex $\Delta$ the following holds. If each subsimplex of $\Delta$ is $\delta$-semitransverse to $\folconst{V}$ in one of its vertices, 
	then there exists $\zeta>0$, namely 
	\[\zeta = C \min \left\{ \delta,\rmin(\Delta),\frac{\delta}{\maxcoeff(\Delta)\rmax(\Delta)} \right\},\] 
	such that that each subsimplex of $\Delta$ remains transverse under simultaneous perturbations of all its vertices by at most $\zeta$.
\end{lemma}
\begin{proof}
	Let $\Delta$ be given as in the statement. We construct a sequence $(\zeta_d)_{d\in\N}$ by induction on $d$ such that each $d$-subsimplex of $\Delta$ remains transverse under simultaneous $\zeta_d$ perturbations of all of its vertices. Using \cref{rem:TransDim}, we see that we can take $\zeta_d = \zeta_{n-k}$ for $d\geq n-k$. 
	
	We start with the case $d=1$, for which we take $\zeta_1 = \delta/2$.

	We now let $d>1$. Let $\Delta_d$ be a $d$-subsimplex of $\Delta$. Then there exists a vertex $v$ and subsimplex $\Delta_{d-1}$ of $\Delta_d$ such that $\Delta_d$ is $\delta$-semitransverse in $v$ and such that $\join{v}{\Delta_{d-1}} = \Delta_d$. By induction, the simplex $\Delta_{d-1}$ is transverse, and remains transverse under simultaneous perturbations of all of its vertices by at most $\zeta_{d-1}$.
	
	By \cref{lem:semitransDiffSimplex} it follows for any $(d-1)$-simplex $\Delta'_{d-1}$ satisfying
	\begin{enumerate}[label=(\arabic*)]
		\item $\Delta'_{d-1} \subset B(p,2\rmax(\Delta))$,
		\item $\Delta'_{d-1}$ is transverse to $\folconst{V}$, and
		\item \label{it:PlaneGrDist} $\dproj(\Gr(\Delta_{d-1}),\Gr(\Delta'_{d-1})) < \frac{\delta}{4\rmax} $, 		
	\end{enumerate}
	that $\join{v}{\Delta'_{d-1}}$ is $\delta/2$-semitransverse in $v$. By \cref{lem:PerturbVertToNbhGr} we obtain that there exists $D_{d-1}$, only depending on $n$ and $d$, such that if we perturb all vertices of $\Delta_{d-1}$ by at most $\beta_{d-1} \coloneqq \delta / (4 D_{d-1} \maxcoeff(\Delta) \rmax(\Delta))$ simultaneously, the resulting simplex $\Delta'_{d-1}$ satisfies \cref{it:PlaneGrDist}. When we define $\zeta_d = \min\{\zeta_{d-1},\beta_{d-1},\rmin(\Delta)/2,\delta/2\}$, we see that $\Delta_d$ is transverse and remains so when we perturb all vertices simultaneously by at most $\zeta_d$.
	
	We now define $\zeta = \zeta_{n-k}$ , which boils down to 
	\[\zeta = \min\left(\frac\delta2,\frac{\rmin(\Delta)}2, \frac{\delta  }{ 4 \max_{d\in\{1,\dots,n-k\}}\{D_{d-1}\} \maxcoeff(\Delta) \rmax(\Delta) } \right).\qedhere\]

\end{proof}

\subsubsection{From semitransversality to \texorpdfstring{$\epsilon$}{ε}-transversality} \label{sec:STToETransPf}

Combining \cref{lem:semiVSQuantTrans,lem:inducSemiTransToPerturbAll} provides us with the following. Here we allow for a scaling by $L$ in the size of the simplex and in its semitransversality, which cancels out because of the quotient in \cref{lem:semiVSQuantTrans}.
\begin{corollary}\label{cor:inducSemiTransToTrans}
	Fix $D\in \N$ and $V\in\Gr(n,k)$. Let real numbers $\delta, \rmin,\rmax,\maxcoeff>0$ be given. Then there exists $\epsilon>0$ such that the following holds.
	
	For any linear $D$-simplex $\Delta$ satisfying that there exists $L \in \R_+$ such that
	\begin{itemize}
		\item $\rmin / L \leq \rmin(\Delta)\leq \rmax(\Delta) \leq \rmax / L$, 
		\item $\maxcoeff(\Delta) \leq \maxcoeff \cdot L$, and
		\item each subsimplex is $\delta/L$-semitransverse to $\folconst{V}$ in one of its vertices, 
	\end{itemize}
	it holds that each subsimplex of $\Delta$ is $\epsilon$-transverse to $\folconst{V}$. 
\end{corollary}
\begin{proof}
	Given such a linear simplex $\Delta$, it follows by \cref{lem:inducSemiTransToPerturbAll} that $\Delta$ is transverse under simultaneous perturbations by at most $\zeta$ of all of its vertices. Here $\zeta $ depends linearly on $L$. Hence by \cref{lem:semiVSQuantTrans}, there exists $\epsilon $ such that $\Delta$ is $\epsilon$-transverse to $\fol$, where $\epsilon$ does not depend on $L$.
\end{proof}

\subsection{Not necessarily constant distributions} \label{sec:nonConstantDistr}
In previous sections, we fixed a constant foliation $\folconst{V}$ and deduced results about the transversality of a simplex to $\folconst{V}$. In this section we develop the tools to work with (not necessarily constant) distributions. We reduce their study to the study of constant foliations by the following observation: a distribution $\xi$ on $\R^n$ induces constant foliations $\folconst{\xi_x}$ for each $x\in \R^n$.

\subsubsection{Semitransversality when varying the foliation}
We start by considering the situation where a simplex is semitransverse to a constant foliation $\folconst{V_1}$, and ask whether we are then able to deduce something about the semitransversality of the simplex to another constant foliation $\folconst{V_2}$. It turns out that if the planes $V_1$ and $V_2$ are close enough, the simplex remains semitransverse, albeit less.

The proof below is similar to \cref{lem:semitransDiffSimplex}, where we use \cref{cor:StProjectSimplex} instead of \cref{cor:semitransEmptyIntQuotient}. Here we use again the distance $\dproj$ on the Grassmannian, derived from the operator norm (recall \cref{def:MetricGrProj}). 

\begin{lemma}\label{lem:semitransDiffFol}
	Let $\beta,\gamma,r>0$ and planes $V_1,V_2 \in \Gr(n,k)$ be given such that $\dproj(V_1,V_2) < \beta$. Let $p$ be a point in $\R^n$ and $\Delta $ a linear $d$-simplex in $\R^n$ such that
	\begin{itemize}
		\item $ \Delta \subset B(p,r)$,
		\item $\Delta$ is transverse to $\fol(V_i)$ for $i=1,2$, and
		\item $\join{p}{\Delta}$ is $\gamma$-semitransverse to $\folconst{V_1}$ in $p$.
	\end{itemize}
	  Then the simplex $\join{p}{\Delta}$ is $(\gamma - \beta r)$-semitransverse to $\folconst{V_2}$ in $p$. 
\end{lemma}
\begin{proof}
	By \cref{rem:TransDim} we can assume that $d+1 \leq n-k$ and without loss of generality we also assume that $p = 0$. We denote by $\pi_i:\R^n \to \R^n / V_i $ the projections for $i=1,2$. Throughout the proof we identify the $V_i$ with their respective, associated linear subspaces of $\R^n$. 
	Additionally, we define $P_i:\R^n \to \R^n$ as the orthogonal projection onto $V_i$. It then holds by assumption that $\|P_1-P_2\|<\beta$.
	
	Since $\join{p}{\Delta}$ is $\gamma$-semitransverse to $\folconst{V_1}$ in $p$, we have that 
	\[ B(0,\gamma) \cap \pi_1(\aspan(\Delta)) = \emptyset \]
	by \cref{cor:semitransEmptyIntQuotient}. This implies that $d(\aspan(\Delta),V_1) > \gamma$, which should be interpreted as the distance between two subsets of $\R^n$. We observe that 
	\begin{align*}
		d(\aspan(\Delta),V_1) = \inf_{v \in \aspan(\Delta)} \|v-P_1 v\| &\leq \inf_{v \in \aspan(\Delta)} \left( \|v-P_2 v\| + \|P_2 v-P_1 v\|\right) \\  &\leq d(\aspan(\Delta),V_2) + \beta \inf_{v\in \aspan(\Delta)} \|v\|,
	\end{align*}
	which implies that \[d(\aspan(\Delta),V_2) \geq d(\aspan(\Delta),V_1) - \beta \inf_{v\in \aspan(\Delta)} \|v\| > \gamma - \beta r. \] Hence the claim follows.
\end{proof}

\subsubsection{Making distributions (almost) constant}
To reduce dealing with transversality to an arbitrary distribution to an almost constant distribution, we state the following. We obtain in particular that, if we consider a set of simplices with small enough diameter, the distribution is almost constant on each simplex. 
\begin{lemma} \label{lem:folAlmostConst}
	Let $\beta>0$ be given, together with a distribution $\xi$ on $\R^n$ and a compact subset $A \subset \R^n$. Then there exists $r>0$ such that $\dproj(\xi_x,\xi_y)<\beta$ for all $x,y\in A$ with $d(x,y)<r$. 
\end{lemma}
\begin{proof}
	Let $\beta>0$ be given, fix $z\in A$ and consider the open ball $\Bproj(\xi_z,\beta) \subset \Gr(n,k)$. Then its preimage $\xi^{-1} (\Bproj(\xi_z,\beta)) \subset \R^n$ is open and hence contains a ball $B(z,r_{\beta,z})$ for some $r_{\beta,z}>0$. By compactness of $A$ the minimum radius $r$ when varying $z$ is attained.
\end{proof}

\subsubsection{Consequences for general position}

In \cref{sec:STToETrans} we considered the situation where a simplex and each of its subsimplices is $\delta$-transverse with respect to a constant foliation. We now deduce what happens to the semitransversality when we replace the constant foliation by a distribution that is not necessarily constant. 

\begin{corollary} \label{cor:semitransAlmostConstant}
	Let $\delta>0$ be given, together with a distribution $\xi$ on $\R^n$ of rank $k$. Then there exists $\rmax>0$ such that the following holds.
	
	For any linear simplex $\Delta$ in $\R^n$ such that
	\begin{itemize}
		\item $\rmax(\Delta)\leq \rmax$, and
		\item each subsimplex $\Delta'$ of $\Delta$ is $\delta$-semitransverse to $\fol(\xi_{x_1})$ in a vertex $v_{\Delta'} \in \Delta'$ for some $x_1\in B(\Delta,\rmax)$,
	\end{itemize}
	it holds that each subsimplex $\Delta'$ of $\Delta$ is $\delta/2$-semitransverse to $\fol(\xi_{x_2})$ in $v_{\Delta'}$ for all $x_2\in B(\Delta,\rmax)$. 	
\end{corollary}
\begin{proof}
	Since semitransversality is stronger than transversality, we know in particular that each subsimplex $\Delta'$ of $\Delta$ is transverse to $\folconst{\xi_{x_1}}$. Hence we can apply \cref{lem:semitransDiffFol}, which tells us that if $\xi_{x_2} \in \Gr(n,k)$ satisfies that $\dproj(\xi_{x_1},\xi_{x_2}) < \delta/(2\rmax)$, the simplex $\Delta'$ is $\delta/2$-semitransverse to $\folconst{\xi_{x_2}}$ in $V_{\Delta'}$.  \Cref{lem:folAlmostConst} provides us with $\rmax\in(0,1)$ such that $\dproj(\xi_{x_1},\xi_{x_2}) < \delta/2$ holds, which implies that $\dproj(\xi_{x_1},\xi_{x_2}) < \delta/(2\rmax)$, for all $x_1,x_2 \in \R^n$ such that $d(x_1,x_2) \leq 3\rmax$. This proves the claim. 
\end{proof}

In particular, we obtain that each subsimplex $\Delta'$ of $\Delta$ is transverse to $\xi_x$ for each $x\in \Delta$. This is equivalent to being in general position, and hence we obtain the following.

\begin{corollary} \label{cor:genPosAlmostConstant}
	Let $\delta>0$ be given, together with a distribution $\xi$ on $\R^n$ of rank $k$.  Then there exists $\rmax>0$ such that the following holds.
	
	For any linear simplex $\Delta$ in $\R^n$ such that
	\begin{itemize}
		\item $\rmax(\Delta)\leq \rmax$, and 
		\item each subsimplex $\Delta'$ of $\Delta$ is $\delta$-semitransverse to $\fol(\xi_{x})$ in a vertex $v_{\Delta'} \in \Delta'$ for some $x\in B(\Delta,\rmax)$,
	\end{itemize}
	it holds that $\Delta$ is in general position with respect to $\xi$. 	
\end{corollary}

\section{Perturbing one vertex} \label{sec:PerturbVertex}
The goal for this section is, given a point $p\in\R^n$ and a set $\SD$ of linear simplices, to find a point $p'$ close to $p$ such that the simplices spanned by $p'$ and $\SD$ are transverse. Hence we introduce the following notation: considering a point $p$ in $\R^n$ and a set $\SD$ of linear simplices in $\R^n \setminus\{p\}$, we denote by $\langle p,\SD \rangle$, the set of linear simplices consisting of $\join{p}{\Delta}$ for all $\Delta \in \SD$. To quantify the transversality of the simplices in $\langle p',\SD\rangle$, we bound their semitransversality in $p'$, which we can deduce from how we choose $p'$.

\begin{remark}
	Throughout this section we work in the following setting: We fix natural numbers $k<n$ and foliate $\R^n$ by copies of $\R^k$ as $\R^n = \R^k \times \R^{n-k}$. We denote this foliation by $\fol$. We let $\pi : \R^n \to \R^n / \fol \simeq \R^{n-k}$ be the projection on the last $n-k$ coordinates. This implies in particular that balls of radius $r$ in $\R^n$ will be mapped under $\pi$ to balls of radius $r$ in $\R^{n-k}$. 
	
	By applying a rotation, the results in this section can be generalized to arbitrary constant foliations $\folconst{V}$ determined by a plane $V \in \Gr(n,k)$. The quotient $\pi$ then turns into $\pi_V : \R^n \to \R^n/V $. We leave the details to the reader.
\end{remark}

\begin{remark}
	This section contains three subsections. In each of which we find $p'$ such that the simplices in $\langle p',\SD \rangle$ are semitransverse, in such a way that the results in each subsection improve upon the results in the previous subsection. We discuss these improvements in the subsections themselves, but here we want to point out that these three subsections are structured similarly: We recall that \cref{cor:semitransEmptyIntQuotient} allows us to check semitransversality of $\join{p}{\Delta}$ in $\R^n$, by checking whether $p$ and $\Delta$ intersect after projecting to $\R^{n-k}$ via $\pi$. Hence each section starts with a result (\cref{lem:deltaPlanes,lem:deltaIndepPlanes,lem:deltaIndepPlanesMultiple}) on points avoiding planes after projecting. Next, we use this to show (\cref{lem:transSimplicesDep,lem:transSimplices,lem:transSimplicesFace}) that we can find $p'$ such that simplices in $\langle p',\SD\rangle$, of a given dimension $d$, are semitransverse in $p'$. The third result (\cref{cor:transPolyDep,cor:transPoly,pro:transPolyInducFace}) uses induction to show that 
	 $p'$ exists such that simplices in $\langle p',\SD\rangle$ of \emph{all} dimensions are semitransverse in $p'$. 
\end{remark}

\subsection{Sets of simplices} \label{sec:polyhedraperturb}
\Cref{cor:semitransEmptyIntQuotient} allows us to check whether a \emph{single} simplex is semitransverse. We now use \cref{lem:transSimplex} to make a \emph{set} of linear simplices semitransverse. That is, we assume that we are given a point $p$ and a set $\SD$ of simplices, and we want to perturb $p$ to $p'$ such that the simplices spanned by $p'$ and the simplices in $\SD$ are transverse. By \cref{lem:transSimplex} it suffices to choose $p'$ such that it does not lie in the affine span of any of the simplices in $\SD$ after taking the projection. This is always possible by the following lemma.

\begin{lemma} \label{lem:deltaPlanes}
	Let $\epsilon>0$ and $d,C \in \N$ be given such that $d<n$. Then for any $p\in \R^{n}$ and any set of $d$-planes $\{F_1,\dots,F_C\}$ in $\R^{n}$, there exists $\delta>0$ and $p'\in\R^{n}$ satisfying that $B(p',\delta) \subset B(p,\epsilon)$ and $B(p',\delta) \cap F_i = \emptyset$ for all $i=1,\dots,C$.
\end{lemma} 
\begin{proof}
	This follows from the observation that each plane $F_i$ has zero measure in the ball $B(\pi(p),\epsilon)$, and there are only $C$ of these planes.
\end{proof}

Using the above result, we first achieve in \cref{lem:transSimplicesDep} that all $d$-simplices in $\langle p',\SD \rangle$ are transverse for a fixed dimension $d$, after which we extend this to any dimension in \cref{cor:transPolyDep}. Hence, with this generalization in mind, the set $\SD$ of linear simplices could contain simplices of any dimension in both statements. We have illustrated \cref{lem:transSimplicesDep} in \cref{fig:transedgesball}.

\begin{lemma}  \label{lem:transSimplicesDep}
	Let $\epsilon>0$ and $d,C \in \N_{>0}$ be given, together with a point $p\in \R^n$ and a set $\SD$ of $C$ linear simplices in $\R^n$ that are transverse to $\fol$. Then, there exists $\delta>0$ and $p'\in B(p,\epsilon)$ such that 
	\begin{itemize}
		\item $B(p',\delta) \subset B(p,\epsilon)$, and
		\item each $d$-simplex in $\langle p', \SD \rangle$ is $\delta$-semitransverse to $\fol$ in $p'$.
	\end{itemize}
\end{lemma}
\begin{figure}[h]
	\includegraphics[width=0.5\textwidth,page=5]{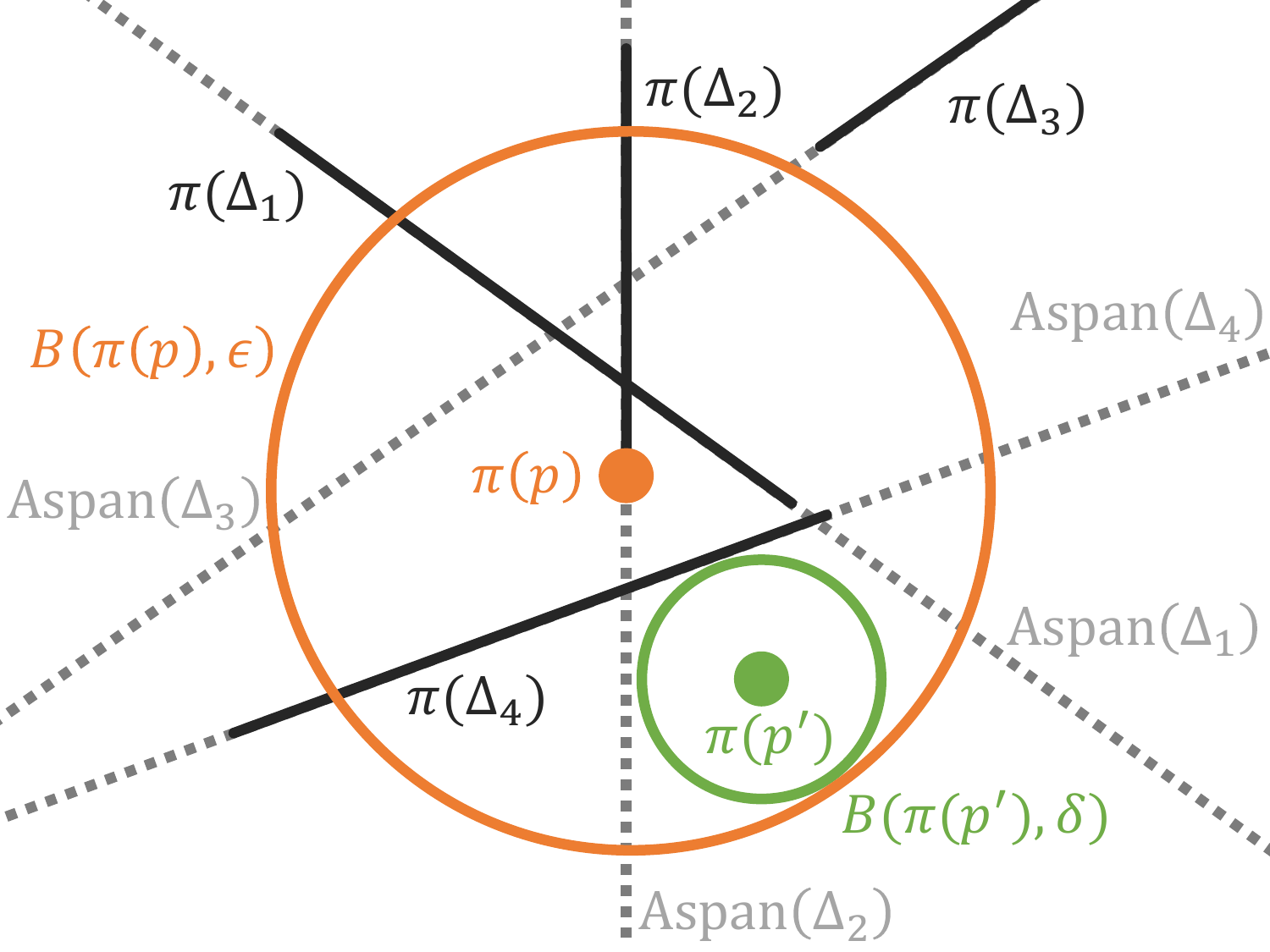}
	\centering
	\caption{A sketch of the proof of \cref{lem:transSimplicesDep} after projecting to $\R^{n-k}$. In orange we have indicated $\pi(p)$ and the $\epsilon$-ball around it. In green we see the point $\pi(p')$ with a $\delta$-neighborhood that does not intersect any of the $\aspan(\Delta_i)$ and is contained in the orange ball.} \label{fig:transedgesball}
\end{figure}

\begin{proof}
	Since, by \cref{rem:TransDim}, the transversality of the $(n-k)$-simplices implies the transversality of the higher dimensional simplices in $\langle p',\SD \rangle$, we can assume that $d\leq n-k$.
	
	By \cref{lem:deltaPlanes} there exist $\hat{p} \in B(\pi(p),\epsilon) \subset \R^{n-k}$ and $\delta>0$ such that $B(\hat{p},\delta) \subset B(\pi(p),\epsilon)$ and $B(\hat{p},\delta) \cap \pi(\aspan(\Delta)) = \emptyset$ for all $(d-1)$-simplices $\Delta$ in $\SD$. 
	 We now define $p'\in\R^n$ as $p' = (p_1,\dots,p_k,\hat{p}_1,\dots,\hat{p}_{n-k})$. Then we have $d(p',p) = d(\hat{p},\pi(p)) < \epsilon$. For the $\delta$-semitransversality in $p'$, we note that if $p'' \in B(p',\delta)$, this implies that $\pi(p'') \in B(\hat{p},\delta)$ and hence by \cref{lem:transSimplex} the claim follows.
\end{proof}

To make the simplices in $\langle p',\SD \rangle$ of any dimension transverse, we repeatedly apply \cref{lem:transSimplicesDep} and obtain the following corollary.

\begin{corollary} \label{cor:transPolyDep}
	Let $\epsilon>0$ and $C\in \N$ be given, together with a point $p\in \R^n$ and a finite set $\SD$ of linear simplices in $\R^n$ that are transverse to $\fol$. Then, there exists $\delta>0$ and $p'\in B(p,\epsilon)$ such that
	\begin{itemize}
		\item $B(p',\delta) \subset B(p,\epsilon)$, and
		\item each simplex in $\langle p',\SD\rangle$ is $\delta$-semitransverse to $\fol$ in $p'$.
	\end{itemize}  
\end{corollary}
\begin{proof}
	Using \cref{lem:transSimplicesDep} we inductively construct a sequence $(p^i)_{i = 1,\dots,n-k}$ of points in $\R^n$ and a sequence of real numbers $(\delta_i)_{i=1,\dots,n-k}$ such that
	\begin{enumerate}[label=(\arabic*)]
		\item\label{it:ball} $B(p^{i},\delta_i) \subset B(p^{i-1},\delta_{i-1})$ where $\delta_0=\epsilon$, and
		\item\label{it:trans} each $i$-simplex in $\langle p^i,\SD\rangle$ is $\delta_i$-semitransverse to $\fol$ in $p^i$.
	\end{enumerate}
	By combining the two items it follows that each simplex in $\langle p^i,\SD\rangle$ of dimension less than or equal to $i$ is $\delta_i$-semitransverse to $\fol$ in $p^i$. 
	Hence we define $p'=p^{n-k}$ and $\delta = \delta_{n-k}$. Semitransversality of simplices of dimension larger than $n-k$ is implied by the transversality of the $(n-k)$-simplices, as by \cref{rem:TransDim}.
\end{proof}

\subsection{Uniform semitransversality} \label{sec:UniSemiT}
When given a vertex $p$ and a finite set $\SD$ of linear simplices, we showed in \cref{sec:polyhedraperturb} how to perturb $p$ to $p'$ such that the simplices in $\langle p',\SD \rangle$ are $\delta$-semitransverse in $p'$. The amount $\delta$ of semitransversality in particular depends on $p$ and $\SD$. Here we show that $\delta$ can be chosen such that it does not depend on $p$ or $\SD$, but on some more general parameters. 

To this end, we observe that an affine $d$-plane in $\R^n$ is determined by a support vector $v \in \R^{n}$ and a direction $V \in \Gr(n,d)$, albeit not uniquely. Hence the set of $d$-planes in $\R^n$ intersecting a subspace $A \subset \R^n$ can be parameterized by the space $\mathrm{Plane}(A,d) \coloneq A \times \Gr(n,d)$, which is compact if $A$ is.

The core observation is then the following.

\begin{lemma} \label{lem:deltaIndepPlanes}
	Let $\epsilon>0$ and $d,C \in \N$ be given such that $d<n$. Then, there exists $\delta>0$ such that for any $p\in \R^{n}$ and any set of $d$-planes $\{F_1,\dots,F_C\}$ in $\R^{n}$, there exists $p'\in\R^{n}$ satisfying that $B(p',\delta) \subset B(p,\epsilon)$ and $B(p',\delta) \cap F_i = \emptyset$ for all $i=1,\dots,C$.
\end{lemma}
\begin{proof}
	We define a function $\tilde{\delta}:\mathrm{Plane}(\overline{B(\pi(p),\epsilon)},d)^C \to \R_+$. Given an element $(F_1,\dots,F_C) \in \mathrm{Plane}(\overline{B(\pi(p),\epsilon)},d)^C$, the function $\tilde\delta$ outputs the maximum $\delta'$ such that there exists $p'\in \R^n$ satisfying $\overline{B(p',\delta')} \subset B(p,\epsilon)$ and $\overline{B(\pi(p'),\delta')} \cap \pi(F_i) = \emptyset$ for $i=1,\dots,C$. As already observed in the proof of \cref{lem:transSimplicesDep}, the number $\delta'$ exists because the $d$-planes have zero measure in $B(p,\epsilon)$ as $d<n$ and there are only finitely many of them.
	
	We observe that $\tilde\delta$ is invariant under a translation by $-p$ of all planes $F_i$ and $p$, and therefore we can assume without loss of generality that $p=0$. Moreover, we claim that $\tilde\delta$ is lower semicontinuous. To see this, fix a set of $(d-1)$-planes $(F_1,\dots,F_C) \in \mathrm{Plane}(\overline{B(\pi(0),\epsilon)},d)^C$ and take $\delta'< \tilde \delta(F_1,\dots,F_C)$. Then there exists an open neighborhood $U$ of $(F_1,\dots,F_C)$ such that $B(p',\delta') \cap \pi(F'_i) = \emptyset$ for $(F'_1,\dots,F'_C)\in U$.
	
	Since $\mathrm{Plane}(\overline{B(\pi(0),\epsilon)},d)^C$ is compact and $\tilde\delta$ is lower semicontinuous, it follows that $\tilde\delta$ attains a minimum, which we denote by $\delta$. This $\delta$ then satisfies the requirements of the statement.
\end{proof}

\begin{remark} \label{rem:deltaIndepPlanesPlane}
	In \cref{lem:deltaIndepPlanes}, we see that if we are given a $D$-plane $H$ in $\R^{n}$ and $d\in \N$ is such that $d<D$, the point $p' $ can be chosen to lie in $H$ since also then the $d$-planes have measure zero in $B(p,\epsilon) \cap H$. We use this in \cref{sec:sTransFace}.
\end{remark}

Using \cref{lem:deltaIndepPlanes}, we obtain a version of \cref{lem:transSimplicesDep} where we have a uniform bound on the amount of semitransversality. Here we first ensure that simplices in $\langle p',\Delta\rangle$ of a fixed dimension $d$ are transverse.

\begin{lemma} \label{lem:transSimplices}
	Let $\epsilon>0$ and $d,C\in \N_{>0}$ be given. There exists $\delta>0$ such that the following holds.
	
	For any $p\in \R^n$ and any set $\SD$ of $C$ linear simplices that are transverse to $\fol$, there exists $p'\in B(p,\epsilon)$ such that
	\begin{itemize}
		\item $B(p',\delta) \subset B(p,\epsilon)$, and
		\item each $d$-simplex in $\langle p',\Delta\rangle$ is $\delta$-semitransverse to $\fol$ in $p'$.
	\end{itemize} 
\end{lemma}
\begin{proof}
	First we note that we can assume by \cref{rem:TransDim} that $d\leq n-k$. 
		
	By \cref{lem:deltaIndepPlanes}, there exists $\delta>0$ such that for any $p\in \R^{n-k}$ and any set of $(d-1)$-planes $\{F_1,\dots,F_C\}$ in $\R^{n-k}$, there exists $\hat p \in\R^{n-k}$ satisfying that $B(\hat p,\delta) \subset B(p,\epsilon)$ and $B(\hat p,\delta) \cap F_i = \emptyset$ for all $i=1,\dots,C$. Hence given $p$ and $\SD$ as in the current statement, there exists $p'\in\R^{n-k}$ such that $B(p',\delta) \subset B(p,\epsilon)$ and $B(\pi(p'),\delta') \cap \pi(\aspan(\Delta)) = \emptyset$ for all $\Delta \in \SD$. Since the simplices in $\SD$ are transverse, it follows by \cref{cor:semitransEmptyIntQuotient} that the semitransversality condition of the statement is satisfied.
\end{proof}

As in \cref{cor:transPolyDep}, we now want to ensure that simplices of any dimension in $\langle p',\SD\rangle$ are transverse. Compared to \cref{cor:transPolyDep}, we additionally achieve that $\delta$ does not depend on $\SD$. 
\begin{proposition}\label{cor:transPoly}
	Let $\epsilon>0$ and $C\in \N_{>0}$ be given. There exists $\delta>0$ such that the following holds.
	
	For any $p\in \R^n$ and any set $\SD$ of $C$ linear simplices that are transverse to $\fol$, there exists $p'\in B(p,\epsilon)$ such that
	\begin{itemize}
		\item $B(p',\delta) \subset B(p,\epsilon)$, and
		\item each simplex in $\langle p',\SD\rangle$ is $\delta$-semitransverse to $\fol$ in $p'$.
	\end{itemize} 
\end{proposition}
\begin{proof}
	We prove this statement by applying \cref{lem:transSimplices} inductively. More explicitly, we obtain a proof for this statement by replacing \cref{lem:transSimplicesDep} by \cref{lem:transSimplices} in the proof of \cref{cor:transPolyDep}.
\end{proof}

\begin{remark}\label{rem:transPolyScaling}
	We note that in \cref{cor:transPoly} the constant $\delta$ can be taken to linearly depend on $\epsilon$. That is, let $\delta$ be as produced by \cref{cor:transPoly} for $\epsilon=1$ and a given $C\in\N$. Then $\delta'=L \delta$ satisfies the conditions of the statement for $\epsilon=L>0$ and the same constant $C$.
\end{remark}

\subsection{Perturbing within a skeleton}\label{sec:sTransFace}
\Cref{cor:transPoly} is now as independent from input data as we will need and suffices for jiggling in the Euclidean setting as in \cref{th:TjigglingEucl}. For \cref{sec:TjigglingSubdiv} we however need a slightly stronger result, which we discuss in the remainder of this section.

In \cref{sec:TjigglingSubdiv} we want to show that if $p$ lies in the $D$-skeleton of a transverse linear simplex or polyhedron, we can choose $p'$ to also lie in this same skeleton. Effectively this means that we want to choose $p'$ in a given $D$-plane $H$ through $p$.  Additionally, we need to deal in \cref{sec:TjigglingSubdiv} with different foliations and sets of linear simplices arising from each (top-dimensional) face containing $p$.

Choosing $p'$ in a given $D$-plane $H$ entails that we cannot achieve semitransversality for simplices of any dimension $d$, but only for those of at most dimension $D$. An exception to this is the case where the dimension $D$ of the plane $H$ is larger than the corank of the foliation, because then we do obtain semitransversality for simplices of any dimension.

In terms of linear planes, we prove the following result. Here we denote for each linear $k$-plane $V_u$ in $\R^n$ the quotient map $\R^n\to\R^n/V_u \cong V^\perp
$ by $\pi_u$, which maps balls in $\R^n$ of radius $r$ to balls of the same radius in $V^\perp$.

\begin{lemma} \label{lem:deltaIndepPlanesMultiple}
	Let $\epsilon>0$ and $d,C,D,U\in\N$ be given such that $d<\min\{D,n-k\}$. Then, there exists $\delta>0$ such that for any 
	\begin{itemize}
		\item point $p \in \R^n$,
		\item set of linear $k$-planes $\{V_1,\dots,V_U\}$ in $\R^n$, 
		\item collection $\{\SE_1,\dots,\SE_U\}$ of sets of $C$ affine $d$-planes in $\R^{n}$ such that the planes in $\SE_u$ are transverse to $V_u$, and
		\item affine $D$-plane $H$ in $\R^{n}$ transverse to all $V_u$,
	\end{itemize}
	there exists $p'\in \R^{n} \cap H$ such that $B(p',\delta) \subset B(p,\epsilon)$ and $B(\pi_u (p'),\delta) \cap \pi_u(F) = \emptyset$ for all $F \in \SE_u$ and all $u=1,\dots,U$.
\end{lemma}
\begin{proof}
	We let $p$, the linear planes $V_u$, the sets of affine planes $\SE_u$, the affine plane $H$ and the natural number $d$ be given. We first argue that $\delta$ and $p'$ exist as required in this specific setting, then we argue it can be chosen independently of the givens. 
	
	Fixing $u$, we note that $\pi_u(H)$ is a plane of dimension $\min\{D,n-k\}$ since $H$ is transverse to $V_u$, whereas the plane $\pi_u(F)$ is only $d$-dimensional for planes $F \in \SE_u$. Hence by \cref{rem:deltaIndepPlanesPlane} we can find $\delta>0$ and $p' \in B(p,\epsilon) \cap H$ satisfying the requirements of the theorem for a single $u\in \{1,\dots,U\}$. Differently put, there exists a measure zero set $Q_u \subset \R^n$ such that if we choose $p' \in Q_u$ there exists no $\delta>0$ satisfying the requirements. Hence if we choose $p'\in B(p,\epsilon) \setminus (\cup_{u=1}^U Q_u)$ there does exist $\delta>0$ such that $p'$ and $\delta$ satisfy the requirements. 
	
	To see that $\delta$ can be chosen independently of the givens, we observe first of all that $\delta$ is invariant under a translation by $-p$ of all givens so that we can assume that $p=0$. Moreover, each plane $V_u$ is determined by an element of $\Gr(n,k)$, whereas we recall that a $k$-plane intersecting $\overline{B(0,\epsilon)}$ is determined (see \cref{sec:UniSemiT}) by an element in $\mathrm{Plane}(\overline{B(0,\epsilon)},k)$ and $H$ is determined by an element in $\mathrm{Plane}(\{0\},D)$. Hence as in the proof of \cref{lem:deltaIndepPlanes}, we take $\delta$ to be the minimum over all such planes $V_u$, sets $\SE_u$ and planes $H$. 
\end{proof}

We now prove the following variation of \cref{lem:transSimplices} where we ensure that $p'$ lies in a given plane $H$ through $p$, and where we deal with multiple sets of linear simplices and multiple planes (each defining a constant foliation). We first achieve semitransversality for simplices of a fixed dimension $d$. 
\begin{lemma} \label{lem:transSimplicesFace} 
	Let $\epsilon>0$ and $C,D,U\in \N$ be given. There exists $\delta>0$ such that the following holds. 
	
	For any
	\begin{itemize}
		\item point $p\in \R^n$,
		\item set of linear $k$-planes $\{V_1,\dots,V_U\}$ in $\R^n$,
		\item sets $\{\SD_1,\dots,\SD_U\}$ of $C$ linear simplices such that the simplices in $\SD_u$ are transverse to $\fol(V_u)$,
		\item affine $D$-plane $H$ through $p$ and transverse to all $\fol(V_u)$, and
		\item natural number $d\in \N_{>0}$ such that $d\leq D$ if $D<n-k$,
	\end{itemize}
	there exists $p'\in B(p,\epsilon) \cap H$ such that 
	\begin{itemize}
		\item $B(p',\delta) \subset B(p,\epsilon)$, and
		\item each $d$-simplex $\Delta \in \langle p',\SD\rangle$ is $\delta$-semitransverse to $\folconst{V_u}$ for all $u\in\{1,\dots,U\}$ in $p'$.
	\end{itemize} 
\end{lemma}
\begin{proof}
	By \cref{rem:TransDim}, we can assume again that $d\leq n-k$.  
	We let $\pi_u: \R^n\to\R^n/\folconst{V_u}\simeq\R^{n-k}$ denote the quotient maps. By \cref{lem:deltaIndepPlanesMultiple}, there exists $\delta>0$ such that for any of the givens as in the current statement, there exists $p' \in \R^{n-k}$ satisfying that $B(p',\delta) \subset B(p,\epsilon)$ and $B(\pi_u(p'),\delta) \cap \pi_u(\aspan(\Delta)) = \emptyset$ for all $(d-1)$-simplices $\Delta$ in $\SD_u$ and all $u \in \{1,\dots,U\}$. By \cref{cor:semitransEmptyIntQuotient} this implies the claim. 
\end{proof}

We now improve upon \cref{lem:transSimplicesFace} by ensuring that simplices of in a range of dimensions are semitransverse. Generally we cannot achieve this for simplices of all dimensions, as was the case in \cref{cor:transPoly}, since we ask for the perturbation $p$ to lie in a $D$-dimensional plane. As in \cref{lem:transSimplicesFace}, we can only achieve semitransversality for $d$-simplices where $d \leq D$ if $D\leq n-k$. If $D > n-k$, we do achieve semitransversality for all simplices.
\begin{proposition} \label{pro:transPolyInducFace}
	Let $\epsilon>0$ and $C,D,U\in \N$ be given. There exists $\delta>0$ such that the following holds.
	
	For any
	\begin{itemize}
		\item point $p\in \R^n$,
		\item set of linear $k$-planes $\{V_1,\dots,V_U\}$ in $\R^n$, 
		\item sets $\{\SD_1,\dots,\SD_U\}$ of $C$ linear simplices such that the simplices in $\SD_u$ are transverse to $\fol(V_u)$, and
		\item affine $D$-plane $H$ through $p$ and transverse to all $\fol(V_u)$,
	\end{itemize}
	there exists $p'\in B(p,\epsilon) \cap H$ such that for all $d \in \N_{>0}$ satisfying $d\leq D$ if $D\leq n-k$, we have 
	\begin{itemize}
		\item $B(p',\delta) \subset B(p,\epsilon)$, and
		\item each $d$-simplex $\Delta \in \langle p',\SD\rangle $ is $\delta$-semitransverse to $\folconst{V_u}$ for all $u\in\{1,\dots,U\}$ in $p'$. 
	\end{itemize}  
\end{proposition}
\begin{proof}
	For each  $D \in \{1,\dots,n-k\}$, the proof is a minor variation of the proof of \cref{cor:transPoly}, since we just need to ask for points in the sequence $(p^j)_{j=1,\dots,i}$ to additionally lie in $H$, and ask for semitransversality to all $\folconst{V_u}$. Hence for each $D \in \{1,\dots,n-k\}$, we follow the proof of \cref{cor:transPoly} but replace \cref{it:ball,it:trans} of the induction hypothesis by 
	\begin{enumerate}[label=(\arabic*)]
		\item $B(p^i , \delta_i) \subset B(p^{i-1},\delta_{i-1}) \cap H$ where $\delta_0=\epsilon$, and
		\item each $i$-simplex in $\langle p^i,\SD \rangle$ is $\delta_i$-semitransverse to $\folconst{V_u}$ for all $u\in\{1,\dots,U\}$ in $p^i$.
	\end{enumerate}
	We use \cref{lem:transSimplicesFace} to deal with this improved induction hypothesis. This produces a set $\{\delta_D \mid 1 \leq D \leq n-k\} $ and we define $\delta$ as its minimum, using again \cref{rem:TransDim} to deal with the transversality of simplices of dimension higher than $n-k$. 
\end{proof}

\section{Jiggling} \label{sec:Tjiggling}
We now turn to proving our jiggling results. In \cref{sec:TjigglingLin} we first jiggle maps from linear polyhedra to Euclidean space. To deal with the issue that being in general position (or being stratified transverse) is not preserved under subdivision, we introduce the idea of jiggling subdivisions in \cref{sec:TjigglingSubdiv}. This enables us to prove a relative version of jiggling in \cref{sec:TjigglingRel}, and hence the manifold case in \cref{sec:TjigglingMfd} where we jiggle maps from linear polyhedra to a manifold. We end by recovering Thurston's jiggling lemma in \cref{sec:TjigglingTriang}.

\subsection{Jiggling in the linear setting} \label{sec:TjigglingLin}
We start by jiggling a map $f: |K|\to \R^n$. The idea of the proof is that we first linearize $f$ so that we can deal with a piecewise linear function. We subdivide so that the linearization is a good approximation and such that the distribution is almost constant within simplices. Then we inductively perturb the vertices in the image of $f_\ell^\lin$ such that each simplex, spanned by vertices that are already perturbed, is in general position with respect to the constant foliation we obtain by fixing the distribution in a point of the simplex. Since the distribution is almost constant within simplices, we deduce from this that the resulting map is in general position.

We recall from \cref{def:ThJiggling} that $(g,K')$ is a $\gamma$-jiggling of $(f,K)$ if $K'$ is a subdivision of $K$, and if $d_{C^1}(f,g) < \gamma$.

\begin{proposition} \label{th:TjigglingEucl} 
	Consider a finite simplicial complex $K$ and a rank $k$ distribution $\xi$ on $\R^n$. Then given
	\begin{itemize}
		\item $\gamma>0$, and
		\item a piecewise embedding $f:|K| \to \R^n$, 
	\end{itemize}
	there exists a $\gamma$-jiggling $(g,K_\ell)$ of $(f,K)$ such that  
	\begin{itemize}
		\item $g$ is a piecewise linear embedding, 
		\item $\dist0(g,f)<\gamma/2^\ell$, and 
		\item $(g,K_\ell)$ is in general position with respect to $\xi$.
	\end{itemize} 
\end{proposition}

\begin{proof}
	\pfstep{Size of perturbation}
	We recall that we denote by $f^\lin_\ell$ the linearization of $f$ with respect to the $\ell$th crystalline subdivision $K_\ell$ of $K$. By \cref{cor:perturbLin} there exists $\epsilon>0$ such that if we perturb the vertices of $f^\lin_\ell$ by at most $\epsilon/2^\ell$, that the resulting map is a $\gamma$-jiggling of $(f,K)$ and is $\gamma/2^\ell$ close in the $C^0$-metric.

	\pfstep{Rescaling the metric}
	We remind the reader that \cref{lem:rmaxminBound} tells us that the size of the simplices of $K_\ell$, and hence of $(f,K_\ell)$, scales by $2^{-\ell}$. Hence, from now on, we rescale the metrics on both the ambient space of $K$ and on $\R^n$ by $2^\ell$ and denote this by $d_\ell$, where $\ell$ corresponds to the number of crystalline subdivisions of $K$. This simplifies the notation, as we can now for instance refer to the above $(\epsilon/2^\ell)$-perturbation as a perturbation by $\epsilon$ with respect to the rescaled metric. We also rescale all related notions, such as $B(\cdot_1,\cdot_2)$ and $\rmax$, by the same factor $2^\ell$ and denote these respectively by $B_\ell(\cdot_1,\cdot_2)$ and $r_\maxa^\ell$. The reader should also interpret semitransversality with respect to rescaled metric throughout the proof. 
	
	\pfstep{Some notation}
	We now introduce the some notation: We denote the vertices of $K_\ell$ by $\{v_0,\dots,v_{N(\ell)}\}$ for some $N(\ell) \in \N$ and denote their images under $f$ by $\{p_0,\dots,p_{N(\ell)}\}$. We let $P_\ell^{(i)}$ be the set of linear simplices that are the image  under $f_\ell^\lin$ of a simplex in $K_\ell$ with vertices contained in $\{v_0,\dots,v_i\}$. 
		
	\pfstep{Obtaining bounds}
	We recall that $\rmax(\Delta)$ is the maximum distance between a vertex and its opposite face in the simplex $\Delta$. From \cref{lem:rmaxminBound} we obtain $B\in \R_+$ such that $ \max_{\Delta' \in K_\ell} \rmax(\Delta') = B $ for all $\ell$ large enough. If we now define $\rmax$ as $\|f\|_{C_1} \cdot B$ and consider a simplex $\Delta \in (f_\ell^\lin,K_\ell)$, we obtain that $r_\maxa^\ell(\Delta) \leq \rmax $ for all $\ell$ large enough. Here we recall that we use the notation $\Delta \in (f_\ell^\lin,K_\ell)$ for simplices $\Delta$ in the image of $K_\ell$ under $f_\ell^\lin$.

	By \cref{cor:crysboundlink} we bound the size of the vertex link of a vertex $v_i \in K_\ell$ independently of $v_i$ and the number of subdivisions $\ell$. 
	Hence there also exists an upper bound $C \in \N$ such that $| \str(v_i) | \leq C$ for all $i \leq N(\ell)$ and $\ell \in N$.
	
	\pfstep{Amount of semitransversality}
	\Cref{cor:transPoly} tells us there exists a real number $\delta>0$ such that the following holds: For any constant foliation $\fol$ of rank $k$, any $p \in \R^n$ and any set $\SD$ of $C$ linear simplices that are transverse to $\fol$, there exists $p'\in \R^n$ such that 
	\begin{itemize}
		\item $B_\ell(p',\delta) \subset B_\ell(p,\epsilon)$ and
		\item each simplex in $\langle p',\SD\rangle$ is $\delta$-semitransverse to $\fol$.
	\end{itemize}
	We recall from \cref{rem:transPolyScaling} that the achieved amount $\delta$ of semitransversality in \cref{cor:transPoly} can be assumed to linearly depend on $\epsilon$, which allows us to state the above using the rescaled metric. We now use this to inductively perturb the vertices of $(f^{\lin}_\ell,K_\ell)$:
	
	\pfstep{Induction hypothesis} 
	The induction hypothesis, depending on $i \in \N$, is that we can find $p'_0,\dots,p'_i \in \R^n$ such that	
	\begin{enumerate}[label={ (IH\arabic*)}]
		\item\label{it:IHball} $p'_j \in B_\ell(p_j,\epsilon)$ for all $j \leq i$, and
		\item\label{it:IHsemitrans} each simplex $\langle {p'_j}_0, \dots, {p'_j}_d \rangle$ is $\delta$-semitransverse to $\folconst{\xi_{p_d}}$ in $p'_{j_d}$, for all $\langle p_{j_0},\dots,p_{j_d} \rangle \in P_\ell^{(i)}$ with $j_d = \max \{j_0,\dots,j_d\}$. 
	\end{enumerate}
	
	We note that in \cref{it:IHsemitrans}, we consider the semitransversality of the \emph{perturbed} simplex $\langle {p'_j}_1, \dots, {p'_j}_d \rangle$ to the foliation obtained from fixing $\xi$ in the \emph{non-perturbed} point $p_{j_d}$. We also observe that for the base case when $i=0$, it suffices to choose $p'_0=p_0$.	
	
	\pfstep{Induction step} We assume that we have defined $p'_0,\dots,p'_{i-1} \in \R^n$ such that the induction hypothesis is satisfied. We now want to find $p'_i$. 
	
	We let $\SD_i$ be the set of linear simplices $\langle p'_{j_0}, \dots p'_{j_d} \rangle$ where $\langle p_{j_0},\dots,p_{j_d} \rangle \in P_\ell^{(i-1)}$ and $\langle v_{j_0} , \dots, v_{d_j} \rangle \in  \str(v_i)$. By induction we know that each simplex $\Delta$ in $\SD_i$ is $\delta$-semitransverse to $\folconst{\xi_q}$ where $q$ is a vertex of $\Delta$. Hence using the existence of the bound $\rmax$, we can assume by \cref{cor:semitransAlmostConstant} that $\ell$ is large enough so that the distribution $\xi$ is almost constant. It follows that each simplex in $\SD_i$ is transverse to $\folconst{\xi_{p_i}}$ in $q$. We also know that $\SD_i$ contains at most $C$ simplices.
	
	By definition of $\delta$, there exists $p'_i \in \R^n $ such that $B_\ell(p'_i,\delta) \subset B_\ell(p_i,\epsilon)$ and such that each simplex in $\langle p'_i, \SD_i \rangle$ is $\delta$-semitransverse in $p'_i$. The former implies that $p'_i$ satisfies \cref{it:IHball}, and the latter that \cref{it:IHsemitrans} is satisfied by all simplices in $P^{(i)}_\ell$ containing $p'_i$. Using the induction hypothesis for $i-1$, we see that \cref{it:IHsemitrans} is satisfied for all simplices in $P^{(i)}_\ell$.
		
	\pfstep{End of induction} The induction ends when $p'_{N(\ell)}$ has been constructed. We define $g : |K| \to \R^n$ as the piecewise linear map uniquely determined by requesting that $g(v_i) = p'_i$. We know in particular that $g$ satisfies the distance requirements from the statement.
	
	\pfstep{Allowing $\xi$ to vary} 
	At this point we know that every simplex $\Delta$ that is in the image of $g$ is $\delta$-semitransverse to $\folconst{\xi_{q}}$ in one of its vertices $q$. That $g$ is hence in general position follows from \cref{cor:genPosAlmostConstant}.
\end{proof}

\begin{remark} \label{rem:TjigglingAlongSubdiv}
	We observe that the above proof only uses the following two properties of the crystalline subdivision $K_\ell$ of $K$:
	\begin{itemize}
		\item there exists $C>0$ such that for all $\ell\in\N$ and vertices $v\in K_\ell$ we have $\vlink(v) \leq C$, and
		\item there exists $\rmax>0$ such that for all $\Delta \in K_\ell$ we have $ r_\maxa(\Delta) \leq \rmax\cdot 2^{-\ell}$.
	\end{itemize}
	Hence if we have another sequence $(\tilde K^\ell)_{\ell\in \N}$ of subdivisions of $K$ for which the above two properties hold, we can apply the above proof and obtain a function $(g,\tilde K^\ell)$ that satisfies the conditions of \cref{th:TjigglingEucl}. We will refer to this as jiggling $(f,K)$ along the subdivisions $(\tilde K^\ell)_{\ell\in \N}$. 
	
	We point out the possibly surprising fact that we do not need any control over the degeneracy of the subdivisions $\tilde K_\ell$ to be able to achieve uniform $\delta$-semitransversality and hence general position. The main reason for this is that \cref{cor:transPoly} produces, based on solely $C$ and $\rmax$, the amount $\delta$ of semitransversality for each simplex. 
\end{remark}

Inspecting the proof of \cref{th:TjigglingEucl} we observe that it does not just produce one map $(g,K_\ell)$ that satisfies the conditions, but instead a sequence of maps $(g_\ell,K_\ell)_{\ell\geq L}$, which become better and better approximations of the given map $f$ with respect to the $C^0$-distance. Moreover, we can bound the transversality of the $g_\ell$ uniformly from below. 

\begin{corollary} \label{cor:towerJiggling}
	Consider a finite simplicial complex $K$ and a distribution $\xi$ on $\R^n$. Then, given
	\begin{itemize}
		\item $\gamma>0$, and
		\item a piecewise embedding $f:|K| \to \R^n$,
	\end{itemize}
	there exists $\epsilon>0$ and a sequence of $\gamma$-jigglings $(g_\ell,K_\ell)_{\ell\geq L}$ of $(f,K)$ such that for all $\ell \geq L$ we have that
	\begin{itemize}
		\item $g_\ell$ is a piecewise linear embedding,
		\item $\dist0(g_\ell,f)<\gamma/2^\ell$,
		\item $(g_\ell,K_\ell)$ is in general position with respect to $\xi$,
		\item each simplex of $(g_\ell,K_\ell)$ is $\epsilon$-transverse.
	\end{itemize}
\end{corollary}
\begin{proof}
	The sequence $(g_\ell,K_\ell)_{\ell\geq L}$ consists of the maps $(g_\ell,K_\ell)$ as produced by the proof of \cref{th:TjigglingEucl} for $\ell$ large enough. It is the last property of the statement that is extra compared to \cref{th:TjigglingEucl}. We note that whereas in the proof of \cref{th:TjigglingEucl} we only used the existence of $\rmax>0$, by \cref{lem:rmaxminBound} there also exist real numbers $\rmin,\maxcoeff>0$ such that for all $\Delta \in K_\ell$ we have
	\begin{align*}
		\rmin / 2^{\ell} \leq r_\mina(\Delta) &\leq r_\maxa(\Delta) \leq \rmax/ 2^{\ell} \text{ and} \\
		\maxcoeff(\Delta) &\leq \maxcoeff \cdot 2^{\ell}.
	\end{align*}
	Hence, since each simplex of each $g_\ell$ is $\delta/2^{\ell}$-semitransverse, \cref{cor:inducSemiTransToTrans} produces the requested $\epsilon>0$ such that the statement holds.
\end{proof}

\begin{remark}\label{rem:TjigglingAlongSubdivSeq}
	Continuing with \cref{rem:TjigglingAlongSubdiv}, we note that we can also jiggle $(f,K)$ along a sequence of subdivisions $(\tilde K^\ell)_{\ell \in \N}$ to obtain a sequence of piecewise smooth maps $(g_\ell,\tilde K^\ell)_{\ell\geq L}$ if	
	\begin{itemize}
		\item there exists $C>0$ such that for all $\ell\in\N$ and vertices $v\in \tilde K_\ell$ we have $\vlink(v) \leq C$, and
		\item there exists $\rmin, \rmax,\maxcoeff>0$ such that for all $\Delta \in \tilde K_\ell$ we have
		\begin{align*}
			\rmin / 2^{\ell} \leq r_\mina(\Delta) &\leq r_\maxa(\Delta) \leq \rmax/ 2^{\ell} \text{ and} \\
			\maxcoeff(\Delta) &\leq \maxcoeff \cdot 2^{\ell}. 
		\end{align*}
	\end{itemize}
	We see that to achieve a uniform amount of $\epsilon$-transversality, in contrast to uniform $\delta$-semitransversality, we do need some control over the degeneracy of $\tilde K_\ell$ in the form of $\rmin$ and $\maxcoeff$.
\end{remark}

\subsection{Jiggling subdivisions} \label{sec:TjigglingSubdiv}
To jiggle a map $f:|K|\to N$, where $N$ is a manifold, we will argue chart by chart in the proof of \cref{th:TjigglingMfd}. This means that we need to implement some form of jiggling relative to a subcomplex $A$ of $K$, over which $f$ is already in general position. However, a crucial part of jiggling is that we subdivide, and if a map $(f|_A,K)$ is in general position or just stratified transverse, this is generally not preserved under subdivision.

Hence we consider the following situation, where a map $f: |K| \to \R^n$ is stratified transverse and we are given a subdivision $K'$ of $K$. Then we cannot assume that $(f,K')$ is in general position, but we can jiggle the subdivision $K'$ such that $f$ is in general position with respect to the resulting subdivision. To make this more explicit, we consider the composition 
\[|K| \stackrel{\id}{\to} |K| \stackrel{f}{\to} \R^n.\]
In \cref{th:TjigglingEucl}, we ensure that $f=f \circ \id$ is in general position by subdividing $K$ and perturbing the map $f$. In this section we ensure that the composition $f\circ \id$ is in general position by subdividing $K'$ and perturbing the map $\id:|K|\to |K| \subset \R^m$. The map $f:|K|\to\R^n$ remains unchanged. The result is a piecewise linear map $T:|K| \to |K| $ that triangulates $K$ and is piecewise linear with respect to $K'_\ell$. Moreover, the map $(f,T)$ is in general position. The result of jiggling a subdivision is illustrated in \cref{fig:jiggleSubdiv}.

\begin{figure}[h]
	\centering
	\begin{subfigure}[b]{0.32\textwidth}
		\centering
		\includegraphics[width=\textwidth,page=6, trim = 4cm 1cm 1cm 0.5cm, clip]{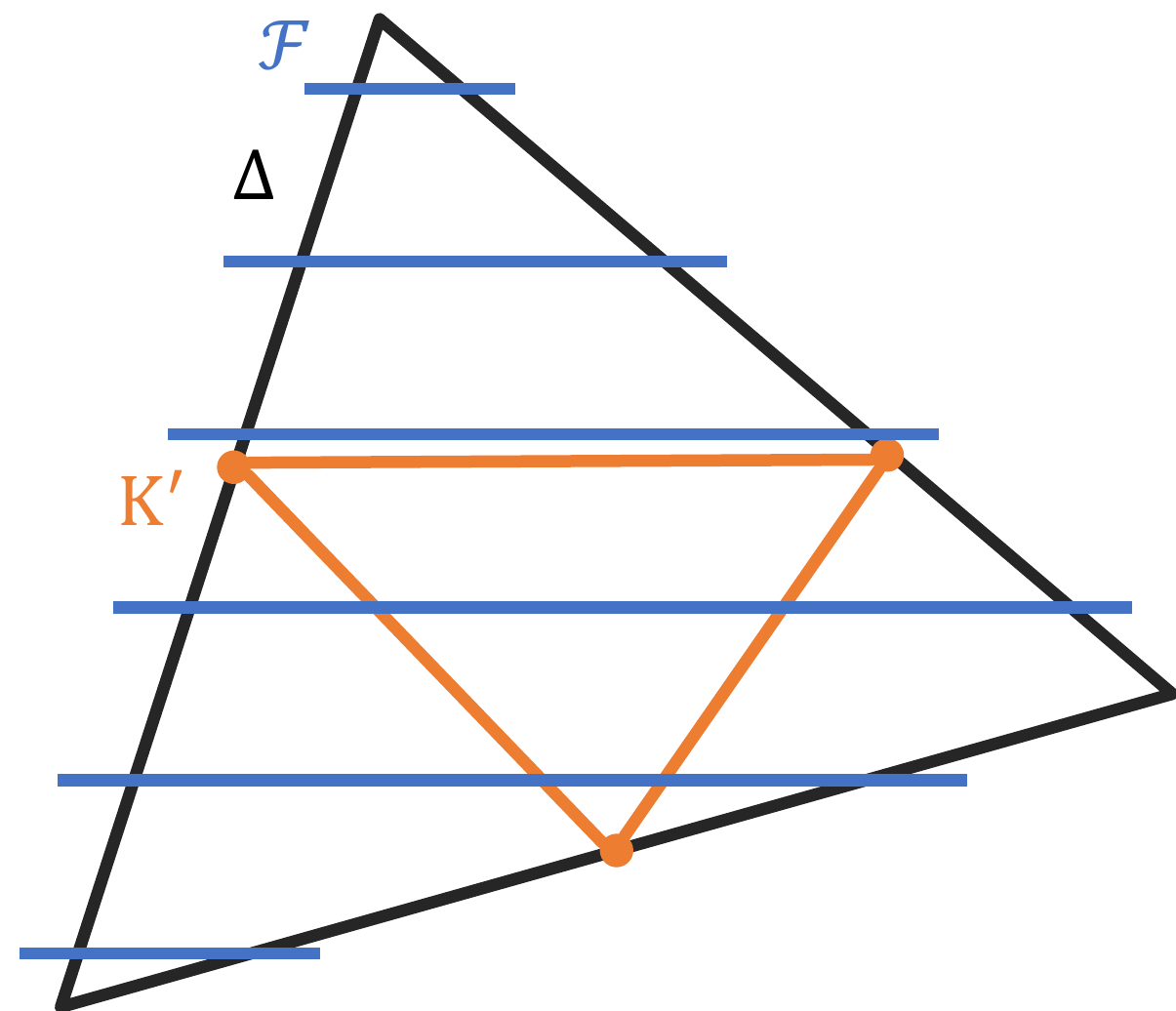}
	\end{subfigure}
	\hfill
	\begin{subfigure}[b]{0.32\textwidth}
		\centering
		\includegraphics[width=\textwidth,page=10, trim = 4cm 1cm 1cm 0.5cm, clip]{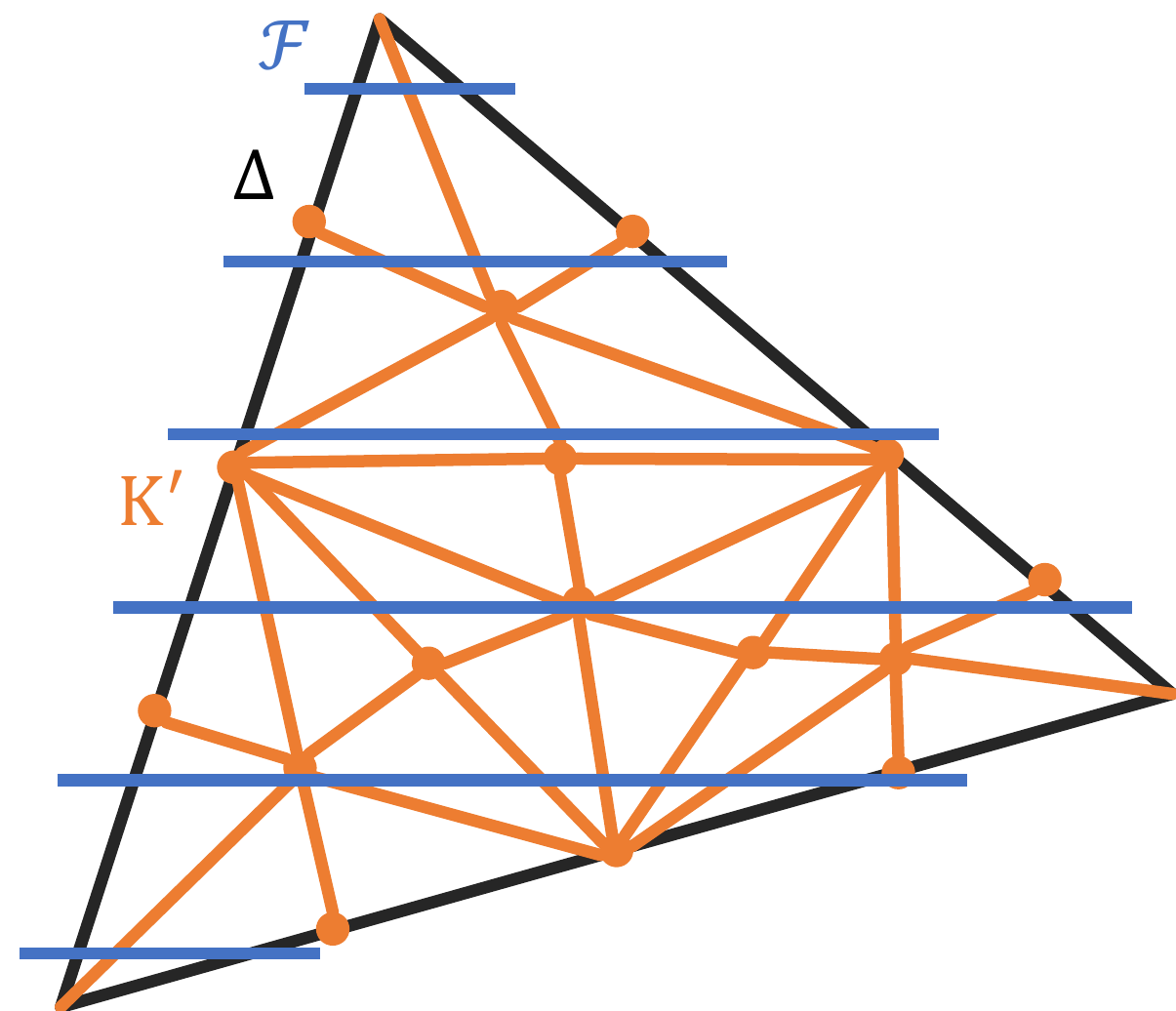}
	\end{subfigure}
	\hfill
	\begin{subfigure}[b]{0.32\textwidth}
		\centering
		\includegraphics[width=\textwidth,page=7, trim = 4cm 1cm 1cm 0.5cm, clip]{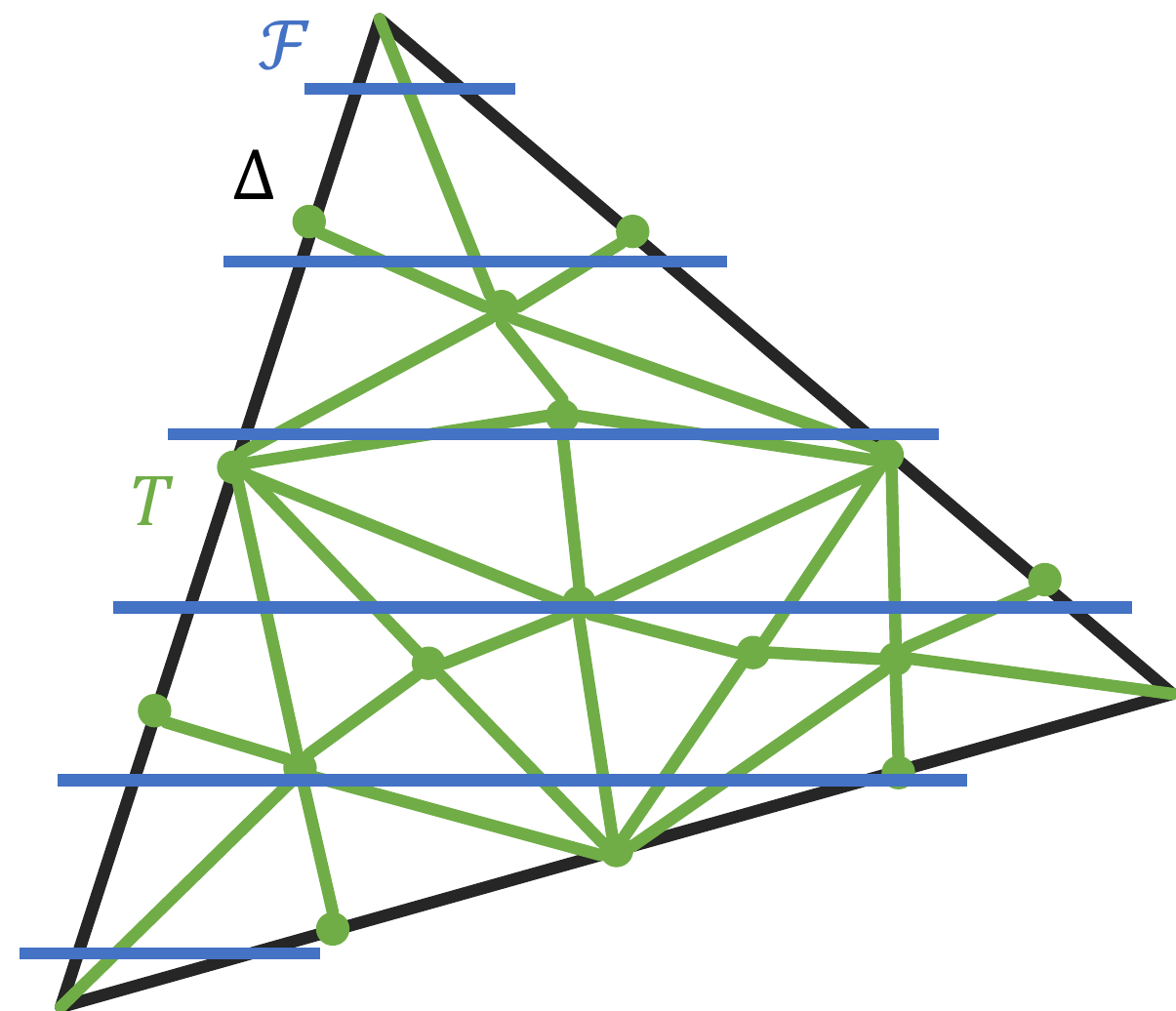}
	\end{subfigure}
	\centering
	\caption{On the left we see a subdivision $K'$ of the simplicial complex consisting of a single simplex $\Delta$. The simplex $\Delta$ is stratified transverse to the horizontal foliation $\fol$, but $K'$ is not. In the middle we indicate the resulting subdivision $K''$ of $K'$ obtained by barycentrically subdividing each simplex of $K'$. On the right we see the result of jiggling the subdivision $K''$, for which we did not use crystalline subdivision. In green we indicate the simplices in the image of $(T,K''_0)$.} \label{fig:jiggleSubdiv}
\end{figure}

We implement the jiggling of a subdivision $K'$ of $K$ using the following lemma. Recall the notation $K^{(\topd)}$ for the top-dimensional simplices of a simplicial complex $K$. We then consider the inclusion $\iota: |K| \hookrightarrow \R^m$ and jiggle $\iota$ such that it is in general position, on each $\Delta \in K^{(\topd)}$, with respect to a distribution defined on that simplex $\Delta$. Moreover, the resulting map $\iota'$ should preserve the skeleta of $K$. Then in \cref{cor:TjiggleSubdivPolyRelativeNoSeq} we see how this achieves the jiggling of a subdivision. 

We note that if $K$ consists of a single top-dimensional simplex $\Delta$, a $\gamma$-jiggling $(\iota',K'')$ satisfying \cref{it:subdivPL,it:subdivGenPos,it:subdivC0} below can be obtained using \cref{th:TjigglingEucl}. Hence compared to \cref{th:TjigglingEucl}, we now additionally need to deal with multiple distributions, and we need to preserve the $j$-skeleton of $K$. 
\begin{lemma} \label{lem:jiggleSubdivIncl}
	Consider a finite simplicial complex $K$ in $\R^m$ and the inclusion $\iota: |K| \hookrightarrow \R^m$. Then, given
	\begin{itemize}
		\item $\gamma>0$,
		\item distributions $\xi_\Delta$ defined on $\Op(\Delta) $ such that $\Delta$ is stratified transverse to $\xi_\Delta$ for all $\Delta \in K^{(\topd)}$, and
		\item a subdivision $K'$ of $K$, 
	\end{itemize}
	there exists a $\gamma$-jiggling $(\iota',K'')$ of $(\iota,K)$ such that
	\begin{enumerate}[label=(\arabic*)]
		\item\label{it:subdivPL} $(\iota',K'')$ is a piecewise linear embedding,
		\item\label{it:subdivC0} $\dist0 (\iota',\id) < \gamma/2^\ell$, 
		\item\label{it:subdivGenPos} $(\iota'|_\Delta,K'')$ is in general position with respect to $\xi_\Delta$ for each $\Delta \in K^{(\topd)}$, and
		\item\label{it:subdivSkeleton} $\iota'$ maps the $j$-skeleton of $K$ to itself, for each $j  \leq\dim(K)$.
	\end{enumerate}
\end{lemma}
\begin{proof} We prove the claim by slightly changing the proof of \cref{th:TjigglingEucl}. We start by barycentrically subdividing each of the simplices $K'$ once, including the non-top-dimensional ones. We denote the resulting subdivison by $\tilde K$.
	
	As in the proof of \cref{th:TjigglingEucl}, we denote the image of the vertices of $K'' \coloneqq \tilde K_\ell$ by $p_0,\dots, p_{N(\ell)}$ for some $N(\ell) \in N$. However, we now order the vertices such that for all $j \in [\dim(K)]$ there exists $S(j) \in \N$ such that the set $\{p_i \mid i \leq S(j)\}$ is the set of vertices lying in the $j$-skeleton of $K$. In the proof of \cref{th:TjigglingEucl} we then inductively perturb each $p_i$ to a $p'_i \in B(p_i,\epsilon)$ such that every simplex in the polyhedron spanned by the vertices up to $p_i$ is transverse to $\folconst{\xi_{p_i}}$. 
	
	We now also perturb the $p_i$ to $p'_i$, but with two extra requirements. Firstly, we want that every simplex in the polyhedron spanned by the vertices up to $p_i$ is transverse to the foliations $\folconst{(\xi_\Delta)_{p_i}}$ for all $\Delta \in K^{(\topd)}$ containing $p_i$. Secondly, we want that if $F$ is the smallest dimensional face of $K$ containing $p_i$, that $p'_i$ is also contained in $F$. By assumption we know that such a face $F$ is transverse to $\folconst{(\xi_\Delta)_{p_i}}$ for all $\Delta \in K^{(\topd)}$ containing $p_i$.
	
	Hence both requirements can be achieved if we use \cref{pro:transPolyInducFace} instead of \cref{cor:transPoly} in the proof of \cref{th:TjigglingEucl}. This is then also the point where the barycentric subdivision and the ordering of the vertices is important: if $p_i$ is contained in a face $F$ of the $j$-skeleton of $K$ (where $j$ is minimal), we need to choose $p'_i$ such that simplices of at most dimension $j$ are made transverse and hence we can choose $p'_i$ to lie in the face $F$ as it is of dimension $j$.
\end{proof}

The jiggling of a subdivision now follows almost directly.

\begin{corollary} \label{cor:TjiggleSubdivPolyRelativeNoSeq}
	Consider a finite simplicial complex $K$ in $\R^m$ and a manifold $N$ endowed with a distribution $\xi$. Then, given
	\begin{itemize}
		\item $\gamma>0$,
		\item a piecewise embedding $f:|K| \to N$ stratified transverse to $\xi$ with respect to $K$,
		\item a subdivision $K'$ of $K$, 
	\end{itemize}
	there exists $\epsilon>0$ and a $\gamma$-jiggling $(T,K'')$ of $(\id,K')$ such that 
	\begin{itemize}
		\item $T$ is a triangulation of $K$, 
		\item $T$ is a piecewise linear embedding,
		\item $\dist0 (T,\id) < \gamma/2^\ell$, and
		\item $(f,T)$ is in general position with respect to $\xi$.
	\end{itemize}
\end{corollary}
\begin{proof}
	For each simplex $\Delta\in K^{(\topd)}$ we extend $f|_\Delta$ to an embedding $f_\Delta$ of an open neighborhood $\Op(|\Delta|)\subset \R^m$ of $\Delta$ into $N$ and consider the preimage $f_\Delta^{-1}(\xi)$ of $\xi$ under $f$. This defines a distribution $\xi_\Delta$ on $\Op(|\Delta|)$ for each $\Delta\in K^{(\topd)}$. Using \cref{lem:jiggleSubdivIncl}, we now jiggle the inclusion of $\iota: |K| \to \R^m$ with respect to the distributions $\xi_\Delta$. We denote the resulting jiggling by $(\iota',K'')$. \Cref{it:subdivSkeleton} of \cref{lem:jiggleSubdivIncl} tells us that $\iota'$ preserves the skeleta of $K$. Hence the map $\iota'$ factors as $\iota \circ T$, where $T$ is the requested triangulation of $K$.
\end{proof}

\subsubsection{Jiggling sequences of subdivisions}
We recall that we introduced the concept of jiggling a subdivision with the goal of jiggling a map $f: |K| \to N$ relative to a subcomplex $A$ of $K$, over which $f$ is already stratified transverse. However, when jiggling $(f,K')$ where $K'$ is a subdivision of $K$, we do a priori not know how often we have to subdivide $K'$. Hence, given a subdivision $K'$ of $K$, it does not suffice to obtain a single jiggling of $K'$ with respect to which $f$ is in general position, but we need a sequence of jigglings.

By combining \cref{cor:towerJiggling} and \cref{cor:TjiggleSubdivPolyRelativeNoSeq}, we obtain the following.
\begin{corollary} \label{cor:TjiggleSubdivPolyRelative}
	Consider a finite simplicial complex $K$ in $\R^m$ and a manifold $N$ endowed with a distribution $\xi$. Then, given
	\begin{itemize}
		\item $\gamma>0$,
		\item a piecewise embedding $f:|K| \to N$ stratified transverse to $\xi$ with respect to $K$,
		\item a subdivision $K'$ of $K$, 
	\end{itemize}
	there exists $\epsilon>0$ and a sequence of $\gamma$-jigglings $(T_\ell,K''_\ell)_{\ell\geq L}$ of $(\id,K')$ such that for all $\ell \geq L$ we have that 
	\begin{itemize}
		\item $T_\ell$ is a triangulation of $K$, 
		\item $T_\ell$ is a piecewise linear embedding,
		\item $\dist0 (T_\ell,\id) < \gamma/2^\ell$, 
		\item $(f,T_\ell)$ is in general position with respect to $\xi$, and
		\item each simplex of $(f,T_\ell)$ is $\epsilon$-transverse.
	\end{itemize}
\end{corollary}

\subsection{Relative jiggling in the linear setting} \label{sec:TjigglingRel}
The sequences of jigglings of \cref{cor:towerJiggling} and the jiggling of a subdivision as in \cref{cor:TjiggleSubdivPolyRelative} enable us to prove a relative version of jiggling. That is, given a piecewise embedding $f: |K| \to \R^n$, we first jiggle the simplicial complex to produce a sequence of subdivisions with respect to which $f$ is in general position when restricted to where it was already stratified transverse. Next, we jiggle $f$ along this sequence of subdivisions to make it in general position where it was not already.
\begin{proposition} \label{lem:TjigglingEuclRel}
	Consider a finite simplicial complex $K$ and a distribution $\xi$ on $\R^n$. Then, given
	\begin{itemize}
		\item $\gamma>0$,
		\item a piecewise embedding $f:|K| \to \R^n$,
		\item a subcomplex $A \subset K$ such that $(f,K)$ is stratified transverse to $\xi$ when restricted to $|\str(A,K)|$, and
		\item a subcomplex $B \subset K$ and an open neighborhood $V \subset |K|$ of $|B|$,
	\end{itemize}
	there exists a $\gamma$-jiggling $(g,K')$ of $(f,K)$ such that
	\begin{itemize}
		\item $g$ is a piecewise linear embedding on $|K|\setminus (|\str(A,K)|\cup V)$,
		\item $(g,K')$ is in general position with respect to $\xi$ on $(|K|\setminus V)\cup |\str(A,K')|$, and
		\item $g|_{|A \cup B|} = f|_{|A \cup B|}$.
	\end{itemize} 
\end{proposition}
\begin{proof}
	Using \cref{lem:niceBarySubdiv}, we see that by barycentrically subdividing $K$ once, the subcomplexes $A$ and $B$ are nice so that we can interpolate over their ring as in \cref{def:interpol}. We denote the resulting subdivision by $K^\bary$. We let $L_1 \in \N$ be such that $|\str(B,K^\bary_{L_1})| \subset V$ and hence take the number of crystalline subdivisions $\ell$ to be at least $L_1$.
		
	Let $\rmin>0$ be such that $\min_{\Delta\in K^\bary_\ell} \rmin(\Delta) = \rmin \cdot 2^{-\ell}$ as in \cref{lem:rmaxminBound}. We jiggle the subdivision $K^\bary$ of $K$ using \cref{cor:TjiggleSubdivPolyRelative}, which produces $L_2 \in \N$ and a sequence $(T_\ell,K''_\ell)_{\ell\geq L_2}$ of triangulations such that $\dist0(T_\ell,\id) < (\rmin/4) \cdot 2^{-\ell}$. Moreover, when we restrict $(f,T_\ell)$ to $|\str(A,K)|$, the map $(f|_{|\str(A,K)|},T_\ell)$ is in general position for all $\ell\geq L_2$ and there exists $\epsilon_f>0$ such that every simplex of $(f|_{|\str(A,K)|},T_\ell)$ is $\epsilon_f$-transverse.
	
	We observe that $(T_\ell)_{\ell \geq L_2}$ is a sequence of subdivisions such that there exists constants $\rmax,\maxcoeff>0$ satisfying that for all $\Delta\in T_\ell$ we have
	\begin{align*}
		(\rmin/2)\cdot 2^{-\ell} \leq r_\mina(\Delta) &\leq r_\maxa(\Delta) \leq (2\rmax)\cdot 2^{-\ell} \text{, and}\\
		\maxcoeff(\Delta) &\leq \maxcoeff \cdot 2^\ell.
	\end{align*}
	Moreover, we bound the maximum size of a vertex link in $T_\ell$ by the maximum size of a vertex link in $v \in K_\ell$. 
	Hence using \cref{rem:TjigglingAlongSubdivSeq}, we jiggle $(f,T_\ell)$ along $(T_\ell)_{\ell}$, and obtain a piecewise linear embedding $(h_\ell,T_\ell)$ such that $(h_\ell,T_\ell)$ is in general position and satisfies that
	\[\dist0(h_\ell,f) < \gamma' /2^\ell \quad \text{and} \quad \dist1(h_\ell,f) < \gamma', \]
	where $\gamma'$, and hence $\ell$, are yet to be determined.
	
	We now define the piecewise linear embeddings $(g_\ell : |K|\to\R^n,T_\ell)$, as
	\begin{itemize}
		\item $h_\ell$ on $|T_\ell \setminus \str(A\cup B,T_\ell)|$,
		\item $f$ on $|A|\cup |B|$, and
		\item the interpolation over simplices in $|\ring(A\cup B,T_\ell)|$.
	\end{itemize}

	It follows that $(g_\ell,T_\ell)$ is in general position over $|T_\ell \setminus \str(A\cup B,T_\ell)|$ and $|A|$ and hence it remains to check that $g_\ell$ is in general position over $\ring(A,T_\ell)$.
	
	Let $\Delta \in \ring(A,T_\ell)$ and let $\Delta_1 \in T_\ell\setminus\str(A,T_\ell)$ and $\Delta_2 \in A$ be two of its faces such that $\Delta = \join{\Delta_1}{\Delta_2}$. 
	We observe that $f = \interpolate{\Delta}{\Delta_1,\Delta_2}{f}{f}$ and $h_\ell|_\Delta = \interpolate{\Delta}{\Delta_1,\Delta_2}{g_\ell}{f}$.
	Hence we can use \cref{lem:interpol} to deduce that $\dist1 (h_\ell|_\Delta,f) = O(\gamma')$. We know that the image of $\Delta$ under $f$ is $\epsilon_f$-transverse. Hence by choosing $\gamma'$ small enough, we obtain that the image of $\Delta$ under $h_\ell$ is $\epsilon_f/2$-transverse. Moreover, if we take $\ell$ large enough, the distribution $\xi$ becomes almost constant as in \cref{lem:folAlmostConst}, implying that the maps $(g_\ell,T_\ell)$ are in general position. 
	
	We conclude the proof by defining $g$ as $g_\ell$, and $K'$ as $T_\ell$, for $\ell$ large enough.
\end{proof}

\subsection{Jiggling in a manifold} \label{sec:TjigglingMfd}
We now move to the case where $f$ is a piecewise embedding of a polyhedron into a manifold $N$. We work chart by chart, relative to previous charts, using \cref{lem:TjigglingEuclRel}. We only indicate the simplicial complex $K$ once in expressions where it occurs multiple times. For instance, we write $\ring(\strT[K]{Q})$ instead of $\ringT[K]{\strT[K]{Q}}$. 

\TjigglingMfd
\begin{figure}[h]
	\includegraphics[width=0.6\textwidth,page=8]{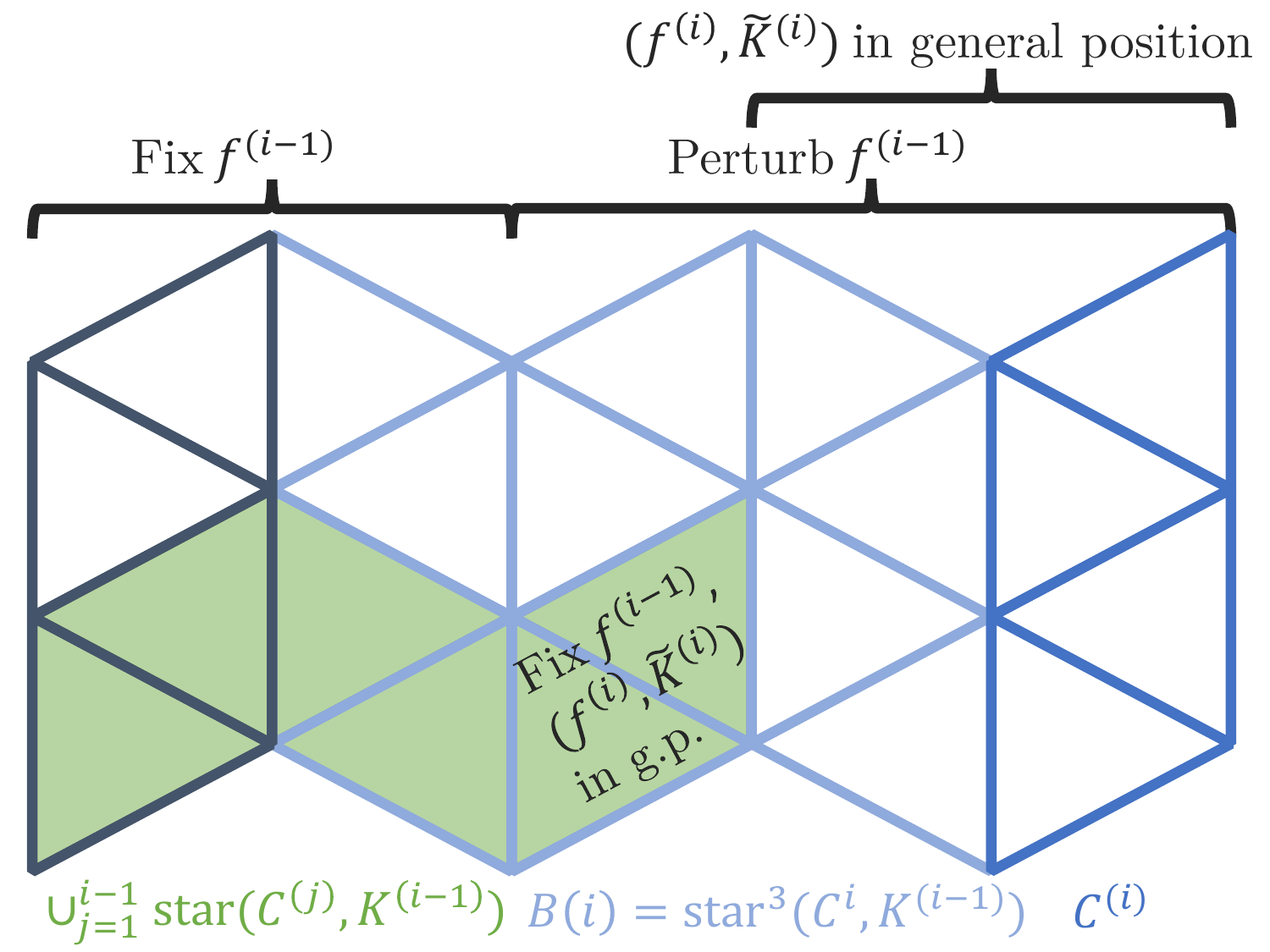}
	\centering
	\caption{A sketch of the proof of \cref{th:TjigglingMfd} when we have constructed $f^{(i-1)}$ and are about to construct $f^{(i)}$. Below we indicate the relevant subcomplexes of $K^{(i-1)}$. In particular, we have indicated with green where $(f^{(i-1)},K^{(i-1)})$ is in general position. At the top we indicate how to obtain $f^{(i)}$ from $f^{(i-1)}$, with as exception the green part where we fix $f^{(i-1)}$, even when the top says otherwise. We also indicate where $(f^{(i)},\tilde K^{(i)})$ is in general position (=g.p.). We point out that this in general does not include the green simplices, which is why we have to jiggle $\tilde K^{(i-1)}$.} \label{fig:TjiggleMfd}
\end{figure}
\begin{proof}
	Suppose first that $A$ is empty. Since the set $B(f(|K|),\gamma) \subset N$ is compact, it can be covered by finitely many charts $\{\phi_i : U_i \to \R^n \}_{i \in [I]}$, where the $U_i$ are opens in $N$ and $I \in \N$. We can assume that $\ell$ is large enough such that we can cover $K_\ell$ by finitely many subcomplexes $\{C^{(i)}\}_{i \in [I]}$ with $B(f(\strN{3}( C^{(i)},K)),\gamma) \subset U_i$.
	
	The argument now amounts to doing induction on $i$. Starting with $f = f^{(-1)}$ and $C^{(-1)} = \emptyset$, we produce a sequence of $\gamma$-jigglings $(f^{(i)},K^{(i)})$ of $(f,K)$ such that $(f^{(i)},K^{(i)})$ is in general position over $|\cup_{j=0}^{i} \str (C^{(j)},K^{(i)})|$. We do so via subsequent applications of \cref{lem:TjigglingEuclRel} as illustrated in \cref{fig:TjiggleMfd}. Concretely, in the $i$th step, we use \cref{lem:TjigglingEuclRel} to jiggle the map $\phi_i \circ f^{(i-1)} : |B(i)| \to \R^n$ which is piecewise smooth with respect to $K^{(i-1)}$ and where $B(i)$ denotes the subcomplex $\strN{3}(C^{(i)},K^{(i-1)})$ of $K^{(i-1)}$.
	 We jiggle $\phi_i \circ f^{(i-1)}$ such that the resulting section $(f^{(i)}:|B(i)|\to N, \tilde K^{(i)})$
	\begin{itemize}
		\item agrees with $f^{(i-1)}$ over $|\ring(\strN{2}(C^{(i)},K^{(i-1)}))|$ and over the intersection of $|B(i)|$ with $|\cup_{j=0}^{i-1} \str(C^{(j)},K^{(i-1)})|$, and
		\item is in general position when restricted to $|\str(C^{(i)},\tilde K^{(i)})|$, and the intersection of $|B(i)|$ with $|\cup_{j=0}^{i-1} \str(C^{(j)},K^{(i-1)})|$.
	\end{itemize} 
	We extend the map $f^{(i)}$ to $|K|$, by requiring that $f^{(i)}$ equals $f^{(i-1)}$ over $|K| \setminus |B(i)|$. We extend the resulting subdivision $\tilde K^{(i)}$ to the entirety of $K$ such that it is a subdivision of $K^{(i-1)}$.
	
	Since $\tilde K^{(i)}$ is a subdivision of $K^{(i-1)}$, the map $(f^{(i)},\tilde K^{(i)})$ is generally not in general position over $|\cup_{j=0}^{i-1} \str(C^{(j)},\tilde K^{(i)}) \setminus B(i) |$. The map $f^{(i)}$ does however equal $f^{(i-1)}$ over this 
	set, where $(f^{(i-1)},K^{(i-1)})$ is in general position. Hence using \cref{cor:TjiggleSubdivPolyRelative}, we jiggle the subdivision $\tilde K^{(i)}$ to a subdivision $K^{(i)}$ such that $(f^{(i)},K^{(i)})$ is in general position when restricted to $|\cup_{j=0}^{i} \str (C^{(j)},K^{(i)})|$. In finitely many steps the proof is complete.
	
	In the relative case where $A \neq 0$ we apply jiggling relative to $A$ throughout the proof by treating it as $K^{(-1)}$.
\end{proof}

\subsection{Triangulations in general position}\label{sec:TjigglingTriang}
We now apply \cref{th:TjigglingMfd} to the identity map $\id: M \to M$ interpreted as a map that is piecewise smooth with respect to a given triangulation $T: |K| \to M$. The resulting jiggled map is a piecewise embedding, and it is a homeomorphism if we take a small enough jiggling \cite[Lecture 4]{lurie2009topics}. However, if $M$ itself is not compact, we can only achieve general position over a given compact. This recovers Thurston's jiggling lemma \cite{Th2}.
\begin{corollary}
	Consider a manifold $M$ triangulated by $T: |K| \to M$ and endowed with a distribution $\xi$. Then given 
	\begin{itemize}
		\item $\gamma>0$, 
		\item a subcomplex $A \subset K$ such that $(T|_{|\str(A)|},\str(A))$ is stratified transverse to $\xi$, and
		\item a compact $C \subset M$,
	\end{itemize}
	there exists a $\gamma$-jiggling $T'$ of $T$ such that 
	\begin{itemize}
		\item $T'$ is in general position with respect to $\xi$ over $C \cup \str(A,T')$, and
		\item $T'|_{|A|} =  T|_{|A|}$.
	\end{itemize}
\end{corollary}

\printbibliography
\end{document}